\documentclass[12pt]{amsart}

\usepackage{amsfonts,amsthm,amssymb}
\usepackage{latexsym,amsmath,setspace}
\usepackage{verbatim,amscd,graphics}
\usepackage{mathrsfs}
\usepackage[pdftex]{graphicx}
\usepackage[matrix,arrow]{xy}
\usepackage[active]{srcltx}
\usepackage{transparent}

\numberwithin{equation}{section}

\usepackage[utf8]{inputenc}
\usepackage{amssymb}
\usepackage{amsmath}
\usepackage{graphicx}
\usepackage{fullpage}
\usepackage{subfig}
\usepackage{url}
\usepackage{enumerate}
\usepackage{mathrsfs}  
\usepackage{mathtools} 
\usepackage{pdflscape} 
\usepackage{multirow} 
\usepackage{listings}
\usepackage{bbding}
\usepackage[font=small,labelfont=bf]{caption}
\usepackage{hyperref}
\usepackage{array}
\usepackage{mathtools}
\usepackage{amssymb}
\usepackage{tikz-cd}
\usepackage{slashed}
\usepackage{physics}
\usepackage{xcolor}
\usepackage[english]{babel}
\usepackage{graphicx}
\graphicspath{ {./images/}}
\usepackage{amsmath}
\usepackage{amsfonts}
\makeatletter
\def\amsbb{\use@mathgroup \M@U \symAMSb}
\makeatother

\newcommand{\CZ}{{\rm CZ}}

\newcommand{\Homeo}{\operatorname{Homeo}}
\newcommand{\Diffeo}{\operatorname{Diffeo}}
\newcommand{\FHomeo}{\operatorname{FHomeo}}
\newcommand{\Cal}{\operatorname{Cal}}

\newtheorem{theorem}{Theorem}[section]
\newtheorem{proposition}[theorem]{Proposition}
\newtheorem{lemma}[theorem]{Lemma}
\newtheorem{corollary}[theorem]{Corollary}
\newtheorem{conjecture}[theorem]{Conjecture}

\newtheorem{question}[theorem]{Question}

\newtheorem{prop}[theorem]{Proposition}

\theoremstyle{definition}

\newtheorem{remark}[theorem]{Remark}
\newtheorem{example}[theorem]{Example}

\newtheorem{exe}[theorem]{Exercise}

\begin{document}

\title[Low-dimensional topology and symplectic dynamics]{Low-dimensional topology and symplectic dynamics}

\begin{abstract}

These are notes to accompany my lectures at the $2024$ ``Current Developments in Mathematics" conference hosted by Harvard/MIT.  The lectures were about some recent progress in our understanding of two and three dimensional dynamical systems, using in part some tools from low-dimensional topology.   In these notes, I try to give a sense for how this works by discussing a few examples from problems I have worked on and surveying some  related developments.  Some symplectic variants of Weyl's law play a key role.   I also briefly comment on some other developments and on the relationship with what is known in higher dimensions.   The lectures themselves are available online.

\end{abstract}

\author{Dan Cristofaro-Gardiner}

\maketitle

\tableofcontents

\section{Introduction}
\label{sec:intro}

\subsection{Overview}

Recent years have seen rapid progress in our understanding of low dimensional symplectic dynamics.   Many longstanding problems have been resolved, powerful new techniques have been developed, and surprising new phenomena have been discovered.  Certain ideas from low-dimensional topology, in particular various three-manifold invariants, have been at the heart of many of these developments.  In these notes, I want to give a sense for how this works.  The notes are meant to be accessible to a broad audience and give some additional detail beyond what I will be able to cover in my lectures.     


\vspace{2 mm}

{\em Disclaimer:}  Low-dimensional symplectic dynamics is a rich and broad area.  For various reasons, these notes are primarily about problems I have worked on 
(though I have attempted to briefly survey some other developments from across the field).  In particular, they focus on certain themes  --- Weyl laws, connections with Seiberg-Witten theory, etc. --- and on recent developments.  I hope that others in the area of low-dimensional symplectic dynamics will write further accounts of parts of the area where these notes are not too detailed.  

\vspace{2 mm}
{\em What is in these notes:}  The heart of the notes is a fairly detailed account of the resolution of two longstanding problems, the ``Simplicity Conjecture" and Hofer-Wysocki-Zehnder's ``two or infinity" conjecture.   I try to discuss these problems in a fair amount of depth, giving a sense of the key ideas that go into the proofs and at least sketching the proofs of many of the key propositions.  I also discuss somewhat more briefly two other topics: the proofs of the two symplectic Weyl laws which play a central role in the argument, and a more recent work about finding invariant sets beyond periodic orbits.   I also highlight some open questions, mention some other recent developments, and try to give some comparison with the situation in higher dimensions.   







\subsection{Terminology}

Let us start by introducing our setting.   A {\em symplectic manifold} is a pair $(M,\omega)$, where $M$ is a  
$2n$-dimensional 
manifold
and $\omega$ is a differential $2$-form such that $d \omega = 0$ and $\omega^n$ is a volume form.  Symplectic manifolds are the natural setting for classical mechanics.  The symplectic form gives the manifold the structure of a ``phase space": one can think of points in $M$ as corresponding to states of some system, and we will return to this below.  As the topic of these notes is the low-dimensional situation, we will be interested primarily in the cases where $n =1$ or $n = 2$.

Symplectic geometry is inherently even-dimensional, but it has an odd-dimensional cousin that will also be of interest to us.  Namely, a {\em contact form} on a $(2n+1)$-dimensional manifold is a differential $1$-form $\lambda$ such that $\lambda \wedge d\lambda^n$ is nowhere vanishing.  Our interest will be in the case $n = 1$.  There is a canonical vector field $R$ associated to a contact form, called the {\em Reeb vector field}, defined by the equations
\[ d \lambda(R, \cdot) = 0, \quad \lambda(R) = 1,\]
and this will also be important for our purposes; we will be particularly interested in the periodic orbits of $R$, called {\em Reeb orbits}.   There is also a canonical hyperplane field $\xi := Ker(\lambda)$ called the {\em contact structure} that will also play an important role.  


\subsection{Two old problems}

We now introduce two concrete and accessible problems.  Explaining how to solve these problems will form the heart of the lectures.


\subsubsection{Mysteries about the group of area-preserving homeomorphisms}

Given a manifold, possibly equipped with some additional structure, an old principle seeks relationships between the relevant group of transformations and the underlying geometry.  
For example, to a topological manifold $M$, one can associate the group $\Homeo_0(M)$, of homeomorphisms in the connected component of the identity.  When $M = S^n$, Ulam asked, in the famous Scottish book, whether this group is simple.  After contributions by many authors, it was known by the $1960$s that for a closed manifold $M$, this group is simple if and only if $M$ is connected.   In other words, the topological property of connectedness is encoded in the algebraic structure of $\Homeo_0$. 

In the $70s$, there was a flurry interest in question of simplicity of other groups of transformations.  For example, if $M$ is a smooth manifold, one can ask whether the group of smooth diffeomorphisms is simple; if $M$ is equipped with a volume form, one can ask about volume-preserving transformations; etc.  By the end of that decade, much was known, see for example the summary in \cite{simp}; a typical picture in these results --- we will see some examples below --- is that either the relevant group is simple, or there are some natural homomorphisms out of the group, leading eventually to a simple kernel.  






However, a curiously stubborn case that had remained open after these works was the case of area-preserving homeomorphisms of surfaces.  For example, the following was asked by Fathi in the 70s \cite{fathi}:

\begin{question}
\label{que:fathi}
Is the group $Homeo_0(S^2,\mu_{std})$ of area and orientation preserving homeomorphisms of the two-sphere simple?
\end{question}

Indeed, when $n \ge 3$, Fathi had shown that the analogous group $\Homeo_0(S^n,\mu_{std})$ of volume-preserving homeomorphisms is simple.
More generally, for any (not necessarily compact) $n$-manifold with suitable measure, Fathi had shown that there is a {\em mass-flow homomorphism}, taking values in a certain abelian group, whose kernel is simple whenever $n \ge 3$ \cite{fathi}.  However, his methods did not extend to dimension $2$.  




\subsubsection{Two or infinity}

Now let us switch to a dynamical problem of a rather different nature, in dimensions $3$ and $4$.  We start with a discussion valid in any dimension before switching to some special conjectures about the low-dimensional situation.

A {\em Hamiltonian} on a symplectic manifold $(M,\omega)$ (of any dimension) is a smooth function $H: M \to \mathbb{R}$.  To any Hamiltonian, there is an associated vector field $X_H$, called the {\em Hamiltonian vector}, defined by the equation
\[ \omega(X_H,\cdot) = dH(\cdot).\]
The dynamics of $X_H$ are of major interest: the flow $\phi^t_H$ of $X_H$ recovers Hamilton's equations of motion, which encode the evolution of a system in classical mechanics.   An easy calculation using Cartan's formula implies the ``conservation of energy" principle that any Hamiltonian flow $\phi^t_H$ must preserve $H$.  Thus, the dynamics occur along level sets $H^{-1}(c)$, which we call {\em energy levels}.

A particularly beautiful dynamical structure is a {\em periodic orbit}; this is a flow line of the system that returns to its initial position.  Thus we are led to a natural question: when does a level set $H^{-1}(c)$ have a periodic orbit?  It is known that examples exist without periodic orbits.
For example, Zehnder \cite{zehnder} has constructed a Hamiltonian flow on $T^{2n}$, $n \ge 2$ (with an irrational symplectic structure) such that the energy levels $H^{-1}(c)$ have no periodic orbits for an entire open interval of energies $c$.  On the other hand, when one imposes the following additional geometric criteria, the situation changes completely, at least conjecturally.  Define a hypersurface $Y$ in $(M,\omega)$ to be of {\em contact type} if there is a contact form $\lambda$ on $Y$ such that $\omega|_Y = d \lambda$.  The following is foundational to the subject:


\begin{conjecture}[``Weinstein conjecture"]
\label{conj:wein}
For any Hamiltonian on any symplectic manifold,
the Hamiltonian flow has a periodic orbit on any compact energy level of contact type.
\end{conjecture}  

When $c$ is a regular value and $Y:= H^{-1}(c)$ is of contact type, with contact form $\lambda$ and corresponding Reeb vector field $R$, an easy computation shows that 
\[ X_H|_Y = f R,\]
 for some positive function $f$.  Hence, the flow of $X_H$ agrees with the Reeb flow up to reparametrization in time.  In particular, Weinstein's conjecture is equivalent to the conjecture that every Reeb flow on a closed manifold with contact form has a closed orbit 
 and it is often stated this way.   When $n = 2$, in other words when $Y$ is a three-manifold, this was proved by Taubes in \cite{Taubes_Weinstein_Conjecture} 
and this can be viewed as a starting point for many of the results we will present here; we will return to Taubes' proof soon.


\begin{remark}
There are an abundance of examples of flows, even volume-preserving ones, with no closed orbits.  Thus, Weinstein's conjecture, if true, gives an example of what can be called ``symplectic rigidity", whereby the symplectic nature of the problem is forcing considerable extra structure. 
\end{remark}



Let us now give an example of all of this that has attracted much interest and that was part of what motivated Weinstein to make his conjecture in the first place.

Let $M = \mathbb{R}^{2n}$, with co-ordinates $(x_1,y_1,\ldots,x_n,y_n)$, and consider the symplectic form $\sum_i dx_i \wedge dy_i.$
(This is called the {\em standard symplectic form}.)  Call a hypersurface {\em star-shaped} when it is smooth and transverse to the radial vector field.  Any star-shaped hypersurface is of contact type: it is an exercise that the one-form $\frac{1}{2}( \sum_i x_i dy_i - y_i dx_i)$ restricts to such a hypersurface as a contact-form.  
The fact that a closed star-shaped hypersurface always has a periodic orbit was first proved by Rabinowitz \cite{rabinowitz} around '78; around the same time, Weinstein also proved this in the special case of boundaries of convex domains \cite{weinstein2}.  Later, Viterbo  \cite{viterbo} proved that any closed contact type hypersurface in $\mathbb{R}^{2n}$ always has a periodic orbit.  Another important result \cite{Hofer93}, by Hofer, proved the Weinstein conjecture for any contact form on $S^3$, or for any contact form on a three-manifold with $\pi_2 \ne 0$, or for any contact form whose associated contact structure is ``overtwisted": this result introduced the method of pseudoholomorphic curves in symplectizations for studying Reeb dynamics, and such pseudoholomorphic curves play a prominent role in our arguments.



One would then like to know to what degree the lower bound of one can be improved.   Once one has a single periodic orbit, one can go around the orbit arbitrarily many times to produce more orbits, so to rule this out we consider {\em simple} orbits i.e. orbits that are not multiply covered.
It is not hard to produce examples with exactly two simple periodic orbits: 

\begin{example}
\label{exe:two}
Let $Y$ be the boundary of the ellipsoid
\[ E(a,b) := \left \lbrace \pi \frac{|z_1|^2}{a} + \pi \frac{|z_2|^2}{b}  \le 1\right \rbrace.\]
This is star-shaped and so inherits a contact form as explained above.  One can compute that, in polar co-ordinates, the Reeb vector field is given by $\frac{2 \pi}{a} \partial_{\theta_1} + \frac{2 \pi}{b} \partial_{\theta_2}$.  Hence when $a/b$ is irrational, there are exactly two simple periodic orbits.
\end{example}

Similarly, one can define irrational ellipsoids in $\mathbb{R}^{2n}$ and these carry precisely $n$ periodic orbits.

From this, it might seem that many other multiplicities are possible.  Returning now to low-dimensions, there is in fact a quite precise prediction from $2001$ contrasting this expectation sharply:





\begin{conjecture}[``Hofer-Wysocki-Zehnder's two or infinity conjecture"]    
\label{conj:hwz}
 
A Hamiltonian flow along any star-shaped hypersurface in $\mathbb{R}^4$ has either two or infinitely many simple periodic orbits.
\end{conjecture}

In fact, generically, it had been known for some time that there were always infinitely many simple periodic orbits for star-shaped hypersurfaces.  Hofer, Wysocki and Zehnder proved their
two or infintiy conjecture
assuming two conditions: 
that the flow is nondegenerate --- we will define this term soon, but it is analogous to a function being Morse --- and assuming that the stable and unstable manifolds of all hyperbolic orbits intersect transversally
\cite{fols}.  So, the challenge in proving it is to develop methods that do not require these conditions.  In fact, one might perhaps think that there are counterexamples if one does not assume these conditions.  For example, for functions on a manifold, one does not generally expect the function with the minimal number of critical point to be Morse. 


 \subsection{Density, periodic orbits, and other invariant sets}
 
 Let us mention some additional results that also relate to this circle of ideas.  
 
 
\subsubsection{Generic density of periodic orbits}

In the previous section, we discussed the question of how many periodic orbits must exist on contact type energy surfaces.  
Another natural question is where such orbits can be found.  In the generic case, there is a beautiful answer to this.  Recall that periodic orbits of the Reeb vector field are called Reeb orbits.


\begin{theorem}\cite{irie}
\label{thm:irie}
A $C^{\infty}$-generic contact form on a closed three-manifold has the property that its Reeb orbits are dense.  
\end{theorem}

In fact, Irie shows in \cite{irie2} that in the $C^{\infty}$-generic case the Reeb orbits even {\em equidistribute}, in the following sense:  there is a sequence of (not necessarily distinct) simple Reeb orbits $\gamma_i$, with corresponding periods $T_i$, such that
\[ \lim_{k \to \infty} \frac{ \gamma_1 + \ldots + \gamma_k}{T_1 + \ldots + T_k} = \frac{d\lambda}{\int \lambda \wedge d \lambda}\]
in the sense of currents.

Sometime after Theorem~\ref{thm:irie}, there was a similar story worked out in the case of surfaces:

\begin{theorem}\cite{cgpz, cgppz, eh}
\label{thm:closing}
A $C^{\infty}$-generic area-preserving diffeomorphism of a closed surface has the property that its periodic points are dense.  
\end{theorem}

The question of generic density of periodic points of area-preserving surface diffeomorphisms had in fact previously attracted considerable interest; see, for example, \cite{fc, xia}.  More generally, questions of this sort are called {\em closing lemmas}.  In the $C^1$-topology, the situation is well-understood due to celebrated work of Pugh and others \cite{pugh, pughrobin}.  However, in higher regularity, the problem is generally considered very hard.  In particular, finding $C^r$ closing lemmas for $r \ge 2$ is the topic of Smale's 10th problem in his list \cite{smale} of problems for the new century; and, if one restricts to Hamiltonian perturbations, counter examples to the smooth closing lemma are known, see \cite{herman2} and the references therein, even for Hamiltonian flows, so it is natural to speculate about what one even expects to be true.  There is now also an equidistribution result \cite{prasad} for area-preserving diffeomorphisms as well.  Theorem~\ref{thm:closing}
adds to 
groundbreaking work of Asaoka-Irie \cite{ai}, who proved the result in the Hamiltonian case.

Irie's work has by now been extensively surveyed, in particular there is a wonderful Bourbaki seminar \cite{bourbaki}.  Thus, we will not say much more here about Irie's work directly --- but we do give a new account of some of the key ideas required to prove the ``contact Weyl law", which is crucial for his arguments; see \S\ref{sec:weyl}.  We also try to give a sense for the Weyl law in the case of surfaces that is required for
the proof of Theorem~\ref{thm:closing} in the same section; we also say a few words about how the arguments in proving generic density go, after one has the appropriate Weyl law, in Remark~\ref{rmk:closing}.







\subsubsection{The Le Calvez - Yoccoz property} 

In the non-generic situation, the periodic orbits need not be dense; see Example~\ref{exe:two}. 

 However, it turns out that if one considers more general invariant sets, there is a very rich picture.  Recall that a flow is called {\em minimal} if every trajectory is dense.  If a flow is not minimal, then it contains a closed {\em non-trivial invariant set}, i.e. a closed proper nonempty set which is invariant under the flow; to construct this set, one takes the closure of a non-dense trajectory.  For example, a Reeb flow on a closed three-manifold is never minimal, since by the Weinstein conjecture it always has a periodic orbit.  As another important example, a compact energy surface of a Hamiltonian on $\mathbb{R}^4$ need not have a periodic orbit \cite{gg}; however, a breakthrough result of Fish-Hofer \cite{fh} that will feature in some of our later discussion shows that nevertheless it can not be minimal\footnote{The regularity here deserves comment.  The counterexamples in \cite{gg} are for $C^2$ Hamiltonians, while the Fish-Hofer result is for smooth Hamiltonians.  It was explained to me by Hofer that the Fish-Hofer arguments should work for $C^{2+\alpha}$.  It seems that the counterexamples in \cite{gg} are expected to hold for $C^{2+\alpha}$ as well}..
 
 The following very recent result guarantees an abundance of closed non-trivial invariant sets without any genericity needed.


\begin{theorem}\cite{cgp}
\label{thm:ley}
Let $Y$ be a closed three-manifold and let $\lambda$ be a contact form on $Y$.  Assume that $c_1(\xi) \in H^2(Y;\mathbb{Z})$ is torsion. The complement of any closed non-trivial invariant set for the Reeb flow is never minimal.  
\end{theorem}

It follows from Theorem~\ref{thm:ley} that under the same hypotheses the union of the closed non-trivial invariant sets is dense and so one obtains a kind of version of Irie's theorem without any genericity.  Theorem~\ref{thm:ley} is inspired by an older theorem, proved by Le Calvez-Yoccoz, proving a similar result in the case of homeomorphisms of $S^2$; see \cite{ley}.

There is a version of Theorem~\ref{thm:ley} for area-preserving maps of surfaces that for brevity we will not state here. 

To bring this more to life, let us consider the example of geodesic flows on closed surfaces.  As has been known for a long time, and as explained in e.g. \cite{cgp}, there is a Reeb flow on the unit tangent bundle whose flow lines project to geodesics.  One can compute \cite[App. A]{cgp} that $c_1$ is torsion for the corresponding contact structure.  Thus, Theorem~\ref{thm:ley} applies.  From this, we can deduce an interesting statement about how much of the surface is seen by the various geodesics.  One could hypothetically imagine that there are situations where essentially every geodesic explores the entire surface.  In fact, we show this is far from the case

\begin{corollary}
\label{cor:nondense}
Any (Riemannian or Finsler) geodesic flow on a closed surface has the property that a dense set of points have a non-dense geodesic passing through them.  Moreover, these geodesics can be assumed to have distinct closures.
\end{corollary}

In negative curvature, Corollary~\ref{cor:nondense} follows from the known fact that in this case, the closed geodesics are themselves dense.  However, 
it seems that this was 
unknown for metrics without curvature assumptions, even in the Riemannian case.

In a different direction, recall from the previous section that there are Hamiltonains on $\mathbb{R}^4$ with 
no periodic orbits at all.  On the other hand, Prasad has shown \cite{prasadnew} that when one works with more general closed non-trivial invariant sets, the dense existence result still holds, on any smooth compact energy surface in $\mathbb{R}^4$.

In these lectures, we will not have a chance to explain the proof of Theorem~\ref{thm:ley} in depth due to time constraints, but we will try to give a flavor for what goes into it in \S\ref{sec:ley}.





\subsection*{Acknowledgements}  I thank the National Science Foundation for
their support under agreements DMS-2227372 and DMS-2238091.  I would also like to thank the organizers of the $2024$ Current Developments in Mathematics conference for inviting me to give these talks.  In these notes, I discuss joint works with a range of collaborators and I would like to thank all of them for numerous crucial conversations.  In particular, the  two works at the heart of these lectures are joint works with V. Humili\`{e}re and S. Seyfaddini; and, with U. Hryniewicz, M. Hutchings, and H. Liu.  Another work that is also an important part of these notes is joint with R. Prasad.  I also thank C. Y. Mak,  D. Pomerleano, I. Smith and B. Zhang for extremely valuable discussions around related joint work.  I would also like to thank G. Forni and S. Newhouse for helping me better understand the context for many of these results.  

I am also extremely grateful to Alberto Abbondandolo, Mohammed Abouzaid, Barney Bramham, Giovanni Forni, Viktor Ginzburg, Helmut Hofer, Umberto Hryniewicz, Vincent Humiliere, Hui Liu, Rohil Prasad, Sobhan Seyfaddini, and Shaoyang Zhou for very helpful comments on an earlier draft of these notes.





\section{The algebraic structure of homeomorphism groups}

We now explain the resolution of Question~\ref{que:fathi}.  We aim to highlight the key ideas.  There is also an excellent Bourbaki seminar \cite{bourbakighys} on this topic that gives another perspective on the work and that we highly recommend.

\subsection{Statement of results}

Recall from the introduction that the group $\Homeo_0(S^n,\mu_{std})$ of volume-preserving homeomorphisms of the $n$-sphere, in the component of the identity, is simple when $n \ge 3$; recall also that the case $n = 2$ had remained a mystery.  Here, $\mu_{std}$ denotes the standard volume measurement.

\begin{theorem}\cite{simp2}
\label{thm:s2case}
The group $\Homeo_0(S^2,\mu_{std})$ of area and orientation preserving homeomorphisms of the two-sphere is not simple.

\end{theorem}  

There is a parallel statement for the open two-disc $D^2$.  Let $\Homeo_c(D^2,\mu_{std})$ denote the group of {\em compactly supported} area-preserving homeomorphisms, namely those that are the identity near the boundary.  (The simplicity question is not interesting for the full group of area-preserving homeomorphism, since $\Homeo_c$ is a normal subgroup.)

\begin{theorem}\cite{simp}
\label{thm:simpconjecture}
The group $Homeo_c(D^2,\mu_{std})$ is not simple.
\end{theorem}  

That the group $\Homeo_c(D^2,\mu_{std})$ is not simple had been called the {\em simplicity conjecture}.  Theorem~\ref{thm:simpconjecture} resolves it.  The proofs of Theorem~\ref{thm:s2case} and Theorem~\ref{thm:simpconjecture} are related.  We will focus on the proof of Theorem~\ref{thm:simpconjecture}, because conceptually it is slightly easier to digest; the argument in the case of the disc also came first.


\subsection{The diffeomorphism case}

The starting point for the proof of Theorem~\ref{thm:simpconjecture} is the analysis in the smooth case.  More precisely, let $\Diffeo_c(D^2,dx\wedge dy)$ denote the group of compactly supported diffeomorphisms, preserving the area form $dx \wedge dy$.  

The following is well-known:


\begin{proposition}
\label{prop:smooth}
The group $\Diffeo_c(D^2,dx\wedge dy)$ is not simple.
\end{proposition}

\begin{proof} [Proof sketch]

We start with the following preliminary from symplectic geometry.  Recall from the introduction that any Hamiltonian $H$ on a symplectic manifold $M$ gives rise to a canonical vector field $X_H$.   If we let our Hamiltonian depend on time, i.e. consider $H: [0,1] \times M \to \mathbb{R}$, we can use the same construction to get a possibly time varying vector field on $D^2$; we will continue to denote it by $X_H$.  An exercise using Cartan's formula shows that the flow $\psi^t_H$ of the vector field $X_H$ (whether time-varying or not) preserves the symplectic form.

Applying this to the case $M = D^2$ thus gives an element of $\Diffeo_c$, for any Hamiltonian.  In fact, it turns out that the converse is true: we can write 
\[ g = \psi^1_H,\]
for any $g \in \Diffeo_c(D^2,dx\wedge dy)$ as the time-$1$ flow of the vector field associated to some (possibly time-varying) Hamiltonian $H$.  Moreover, since $g$ is compactly supported, we can assume in addition that $H = 0$ near $\partial D^2$.

We now define the map
\[ \operatorname{Cal}:    \Diffeo_c(D^2,dx\wedge dy) \to \mathbb{R}, \quad g = \psi^1_H \to \int^1_0 \int_{D^2} H(x,t) dx dt.\] 
A beautiful fact is that this is well-defined.  A priori, this is not clear,
since there are many different Hamiltonians that all have the same time-$1$ flow; however, in fact the relevant integral does not depend on the choice of $H$.

Once we know that $\operatorname{Cal}$ is well-defined, it is not hard to see that it is a homomorphism.  Indeed, there is an easy construction of a Hamiltonian $H \# G$ such that $\phi^1_{H \# G} = \phi^1_H \circ \phi^1_G$ and $\int H \#G = \int H + \int G$, see \cite[\S2]{simp}.    The homomorphism $\operatorname{Cal}$  can not have trivial kernel, since $\Diffeo_c(D^2,dx\wedge dy)$ is not abelian, and by construction it can not be identically zero.  Thus, its kernel gives a non-trivial proper normal subgroup.
\end{proof}

The homomorphism $\operatorname{Cal}$ is called the {\em Calabi homomorphism} and it will play a key role in our story.  It turns out that it measures the ``average rotation" of the map.
As far as the simplicity of $\Diffeo_c(D^2,dx\wedge dy)$, it is the key obstruction: Banyaga has shown \cite{banyaga}
that its kernel is simple, although we will not use that fact here.

\subsection{Difficulties in extending Calabi}
\label{sec:difficulty}

We have an inclusion of groups
\[  \Diffeo_c(D^2,dx\wedge dy) \subset \Homeo_c(D^2,dx\wedge dy.\]
To attempt to prove Theorem~\ref{thm:simpconjecture}, it would now be natural to try to extend 
$\operatorname{Cal}$ from $\Diffeo_c$ to $\Homeo_c$.  In fact, it is well-known that $\Diffeo_c$ sits densely in $\Homeo_c$ in the $C^0$-topology, so one could attempt to extend $\operatorname{Cal}$ by continuity.

However, this can not work:

\begin{lemma}
$\operatorname{Cal}:  \Diffeo_c(D^2,dx\wedge dy) \to \mathbb{R}$ is not continuous in the  $C^0$-topology.  
\end{lemma}

\begin{proof}
Fix $n \le 1/2$.  Let $(r,\theta)$ be polar coordinates on the disc.  Let $H_n$ be a non-negative (time-independent) Hamiltonian on $D^2$ that vanishes for $r \ge 2/n$ and has $\int H_n = 1$ (see Figure~\ref{hi}). Then $\operatorname{Cal}(\psi^1_{H_n}) = 1$.  On the other hand, as $n \to \infty$, $\psi^1_{H_n} \to id$ in the $C^0$-topology, where $id$ denotes the identity.  However, $\operatorname{Cal}(id) = 0$.   
\end{proof}

The challenges of extending Calabi have been well-known for a long time.  Indeed, Fathi's paper \cite{fathi}  asks if it is possible to extend it. 

\begin{figure}[h!]
\caption{The Calabi invariant is not $C^0$-continuous.} 
\label{hi}
\centering
\def\svgwidth{,5 \textwidth}
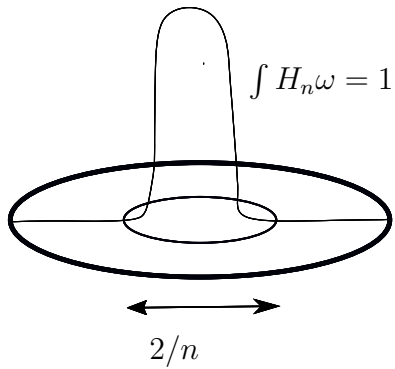
\end{figure}

\subsection{A normal subgroup via Hofer geometry}

If we could extend the Calabi homomorphism, then Theorem~\ref{thm:simpconjecture} would be proved, similarly to the proof of Proposition~\ref{prop:smooth}.  However, as we explained in \S\ref{sec:difficulty}, a priori it is not clear how to do this, so we instead take a more direct approach: we will write down a normal subgroup, containing $\Diffeo_c(D^2,dx \wedge dy)$, and show that it is proper.

To define the right normal subgroup, we make use of a remarkable bi-invariant metric discovered by Hofer.  Let $M$ be any symplectic manifold.  For a compactly supported Hamiltonian $H:C^{\infty}( \mathbb{R}/\mathbb{Z} \times M)$ we define the {\em one-infinity norm}
\[ ||H||_{1,\infty} := \int^1_0 (\operatorname{max}_{x \in M} H(x,t) - \operatorname{min}_{x \in M} H(x,t)) dt\]
and we then define, for a compactly supported (non-autonomous) Hamiltonian diffeomorphism $\phi$ of $M$, the {\em Hofer norm}
\[ ||\phi|| = \operatorname{inf} \lbrace ||H||_{1,\infty} | \hspace{2 mm} \phi = \psi^1_H \rbrace,\]
where the infimum ranges over all compactly supported Hamiltonians.  We now define the {\em Hofer distance} between two compactly supported Hamiltonian diffeomorphisms $\phi_1$ and $\phi_2$ by the formula
\[ d(\phi_1,\phi_2) = || \phi_1 \phi^{-1}_2 ||.\]
It is a fact that this defines a bi-invariant metric on the group of Hamiltonian diffeomorphisms of $M$.  (One should pause for a moment to appreciate this: in finite dimensions such a bi-invariant metric on a non-compact group like this is essentially impossible to come by.). We will not prove this here --- the only part we need is the bi-invariance, which is an exercise; the hardest part is nondegeneracy  

With these preliminaries understood, we now define our normal subgroup.  We define a {\em finite Hofer energy homeomorphism} to be a homeomorphism 
\[ h \in \Homeo_c(D^2,\mu_{std})\] 
such that there exists smooth $g_n \in \Diffeo_c(D^2,dx \wedge dy)$ and a positive constant $C$ with 
\[ g_n \to_{C^0} h, \quad \quad d(g_n,id) \le C.\]
In other words, $h$ is a finite Hofer energy homeomorphism if it can be approximated in $C^0$ by diffeomorphisms, without the Hofer norm blowing up.  We let $\FHomeo_c(D^2,\mu_{std})$ 
denote the set of finite Hofer energy homeomorphisms.

The following is an easy consequence of the bi-invariance of Hofer's metric, whose proof we will not include here.

\begin{proposition} \cite{simp}
$\FHomeo_c(D^2,\mu_{std})$ is a normal subgroup of $\Homeo_c(D^2,\mu_{std})$
\end{proposition}







\subsection{Recovering Calabi with the Weyl law}
\label{sec:calpfh}

We can now explain the basic idea behind the proof of Theorem~\ref{thm:simpconjecture}.
Let $\varphi \in \Diffeo_c(D^2,dx\wedge dy)$.  For each $d \in \mathbb{N}$, we define {\em PFH spectral invariants}
\[ c_d(\varphi) \in \mathbb{R},\]
satisfying the following properties:

\vspace{2 mm}

\begin{itemize}
\item {\em Continuity.} Each $c_d$ is $C^0$ continuous and extends continuously to $\Homeo_c(D^2,dx\wedge dy)$.
\item {\em Montonicity.} For any (possibly time-dependent) Hamiltonians $H, K$, if $H \le K$, then 
\[ c_d(\phi^1_H) \le c_d(\phi^1_K).\] 
\item {\em Identity.} $c_d(id) = 0$, where $id$ is the identity.
\item {\em Hofer Lipschitz.}   For any (possibly time-dependent) Hamiltonians $H, K$
\[ | c_d(\phi^1_H) - c_d(\phi^1_K) | \le d ||H - K||_{1,\infty} .\]
\item {\em Weyl law.} We have
\[ \lim_{d \to \infty} \frac{ c_d(\phi) }{ d} = Cal(\phi).\]
\end{itemize}

We will explain the construction of the PFH spectral invariants in \S\ref{sec:pfh}.  The definition is due to Hutchings.  It is in the construction of these invariants that the connection with three-manifold invariants comes up. 
  
The idea is now to apply the above properties to show that $\FHomeo$ is proper.  More precisely, our strategy is as follows.  We will define a particular homeomorphism  $T$ --- an ``infinite twist" with ``infinite Calabi invariant".  By the Continuity property, each $c_d$ is defined for homeomorphisms as well, and the idea is then to study the asymptotics of $c_d(T)$.  We will use the Weyl Law and Monotonicity properties to show that 
\[ c_d(T)/d \to \infty.\]   
On the other hand, the Hofer Lipschitz property and the Continuity property will be used to show that for any $g \in \FHomeo$, $c_d(g)/d$ remains bounded in $d$.  

\subsection{The infinite twist and the proof of Theorem~\ref{thm:simpconjecture}}
\label{sec:proof}

Let us now make the strategy in the previous section precise.    That is, we explain the construction of the infinite twist and the proof of Theorem~\ref{thm:simpconjecture}, assuming the properties of the PFH spectral invariants from \S\ref{sec:calpfh}.  

We start by defining the ``infinite twist".  We consider radial co-ordinates $(r,\theta)$ on the disc.  We fix a smooth function $f(r):(0,1] \to [0,\infty)$ and define the map so that $T(0) = 0$ and, for $ r > 0$, 
\[ T(r,\theta) = (r, \theta + f(r) ).\]
We now choose $f$ as follows:
\begin{itemize}
\item $f(r) = 0$ for $r$ close to $1$.
\item $f$ is non-increasing.
\item $\int^1_0 \int^\infty_r s f(s) ds dr = \infty.$  
\end{itemize}   
See Figure~\ref{fig:inftw}.  The reason for the first bullet point is that it ensures that $T$ is compactly supported.  The second will be useful in combination with the Monotonicity Axiom.  The point of the third bullet point is to make precise that $T$ has ``infinite Calabi invariant".  Indeed, if $f$ was such that the corresponding twist map was smooth, then one can check that the integral   $\int^1_0 \int^\infty_r s f(s) ds dr$ is precisely the Calabi invariant.

\begin{figure}[h!]
\caption{The infinite twist}
\label{fig:inftw}
\centering
\def\svgwidth{.8\textwidth}
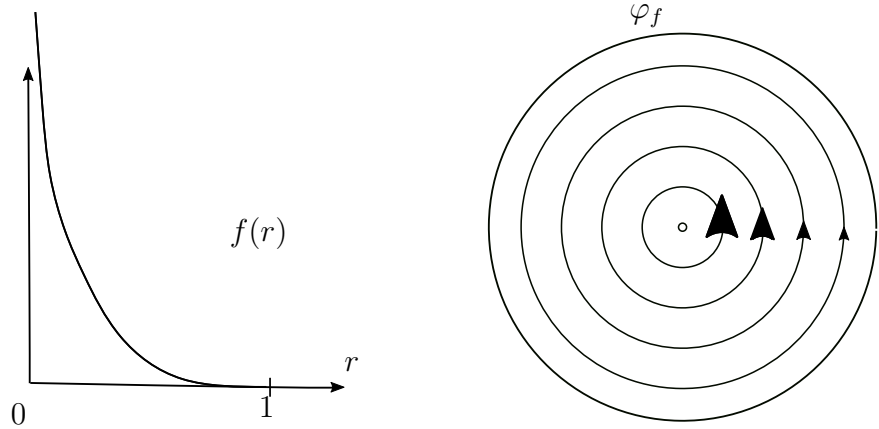

\end{figure}

\begin{proof}[Proof of Theorem~\ref{thm:simpconjecture}, assuming the PFH spectral invariant axioms]

Let $T$ be the infinite twist defined above.  

\vspace{3 mm}

{\em Step 1.  $c_d/d$ diverges:}  Define smooth maps $T_{f_i}$ by $T_{f_i}(r,\theta) = (r, \theta + f_i(r) ).$  That is, the $T_{f_i}$ are twist maps, but unlike the infinite twist, they are smooth.  The twists $T_{f_i}$ depend on the choice of $f_i$ and we choose these $f_i$ as follows: we let 
\[f_i = f\] 
for $r$ in $[1/i, 1]$, and we demand that 
\begin{equation}
\label{eqn:mon}
f_i \le f_{i+1}.
\end{equation}
See Figure~\ref{fig:approx}.
Then:
\begin{itemize}
\item $T_{f_i} \to_{C^0} T$,
\item $\Cal(T_{f_i}) \to \infty$
\item $c_d(T_{f_i}) \le c_d(T_{f_{i+1}}).$
\end{itemize}
The proof of the above three bullet points is straightforward and we leave it as an exercise to the reader.  (For the third point, one needs to combine the Monotonicity Axiom with \eqref{eqn:mon}.)
For any $i$, we now have 
\[ \frac{ c_d(T_{f})}{d} \ge \frac{ c_d(T_{f_i})}{d}\]
as a consequence of the Continuity property of the $c_d$ and the third bullet point above.  Hence, for any $i$, we have
\[ \lim_{d \to \infty} \frac{ c_d(T_{f})}{d} \ge \lim_{d \to \infty} \frac{ c_d(T_{f_i})}{d} = \Cal(T_{f_i}),\]
where in the final equality we have used the Weyl law.  Thus, by the second bullet point, $\frac{ c_d(T_{f})}{d}$ is unbounded in $d$, as desired.

\vspace{2 mm}
 
 {\em Step 2.  For any $g \in \FHomeo, c_d/d$ does not diverge.}  Note that this step, in combination with the previous one, implies Theorem~\ref{thm:simpconjecture}.  
 
 To prove it, we use the Hofer Lipschitz property.  Let $g \in \FHomeo$, and take $f_i \to_{C^0} g$ with $|f_i| \le C$.  Write $f_i = \psi^1_{H_i}$ with 
 \[ ||H_i||_{1,\infty} \le 2C.\]
 Then, by the Hofer Lipschitz and Identity axioms
 \[ | c_d(f_i) | \le 2d C.\]
 By the Continuity Axiom, it now follows that $| c_d(g) | \le 3d C$, as desired.

\end{proof}

\begin{figure}[h!]
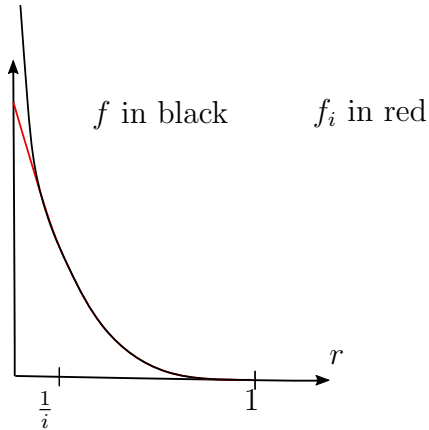

\caption{Approximating the infinite twist by smooth twists}
\label{fig:approx}
\centering
\def\svgwidth{.8\textwidth}

\end{figure}


\section{The PFH spectral invariants and their properties}
\label{sec:pfh}

In this section, we explain the construction of the PFH spectral invariants used in the previous section.  We also give some idea of their properties.  

\subsection{Periodic Floer Homology (PFH)}

Let $(\Sigma,\omega)$ be a closed surface, and $\phi$ an area-preserving map.  For convenience, we will assume $\Sigma = S^2$ --- this suffices for the proof of Theorem~\ref{thm:simpconjecture} --- but what we write has an analogue for any orientable surface (see e.g. \cite{cgpz}).

To explain the construction of the PFH spectral invariants, we first need to explain the construction of PFH.  Let $d \in \mathbb{Z}_{\ge 0}$.  The {\em Periodic Floer Homology} $PFH(S^2,d)$ is a vector space over\footnote{One can define PFH over $\mathbb{Z}$-coefficients, but we will not need this.} $\mathbb{Z}_2$.  It is the  homology of a chain complex $PFC(\phi)$.  Its definition is due to Hutchings.  To define this chain complex, we start by forming the mapping torus 
\[ Y_{\phi} = S^2 \times [0,1]_t/ \sim,  \quad (x,1) \sim (\phi(x), 0). \]
This has a canonical vector field 
\[ R := \partial_t,\]
and a canonical two-form $\omega_\phi$ induced by $\omega$.  We can think of $R$ as analogous to the Reeb vector field, except that it will not be a Reeb vector field associated to a contact structure.  (It can be regarded as the ``Reeb vector field" associated to the pair $(dt,\omega_\phi).$)  In analogy with the contact case, we call a periodic orbit of $R$ a {\em Reeb orbit}.  The Reeb orbits are in bijection with the periodic points of $\phi$.  

Now assume that $\phi$ is {\em nondegenerate}: this means that $1$ is not an eigenvalue of the linearization of $\phi^m$ at any $m$-periodic point.   We call a periodic point {\em hyperbolic} if these eigenvalues are real, and {\em elliptic} if they lie on the unit circle, and we classify the corresponding Reeb orbits as hyperbolic or elliptic accordingly.  We define an {\em orbit set} to be a finite set $\lbrace (\alpha_i,m_i) \rbrace$ where the $\alpha_i$ are distinct simple Reeb orbits and the $m_i$ are positive integers.  We call an orbit set a {\em PFH 
generator} if $m_i = 1$ whenever $\alpha_i$ is hyperbolic.  The {\em degree} of a PFH generator is its image under the natural map $H_1(Y_\phi, \mathbb{Z}) \to H_1(S^1,\mathbb{Z}) = \mathbb{Z}.$  This corresponds to the sum of the periods of the relevant orbits.

The chain complex $PFC(\phi,d)$ is freely generated over $\mathbb{Z}_2$ by degree $d$ PFH generators.  To define the chain complex differential $\partial$, we make use of Gromov's theory of pseudoholomorphic curves.  Namely, let $X = \mathbb{R}_s \times Y_\phi$; we call this the {\em symplectization} of $Y_{\phi}$.  Let $V$ denote the vertical tangent bundle for the fibration $Y_\phi \to S^1$.  Recall that an {\em almost complex structure} $J$ on $X$ is a smooth bundle map $J: TX \to TX$ such that $J^2$.  We call an almost complex structure {\em admissible} if $J$ preserves $V$, 
rotating vectors positively 
with respect to $\omega_\phi$, and sends $\partial_s$ to $R$.  Let $\mathcal{M}_J(X,\alpha,\beta)$ denote the space of $J$-holomorphic curves in $X$, namely maps
\[ u : (\Sigma,j) \to (X,J), \quad du \circ j = J \circ du,\]
where $(\Sigma,j)$ is a compact Riemann surface minus finitely many punctures, that are asymptotic to $\alpha$ at $+\infty$ and asymptotic to $\beta$ at $-\infty$; see Figure~\ref{pfhd}. We regard two $J$-holomorphic curves as equivalent if they are equivalent as currents, and we refer to an element of $\mathcal{M}$ as a $J$-{\em holomorphic current}.  
We now choose $J$ generically, so that all relevant moduli spaces we want to consider are manifolds, and define
\[ \partial \alpha = \sum_\beta \langle \alpha, \beta \rangle \beta,\]
where 
\[ \langle \alpha, \beta \rangle = \# \mathcal{M}^{I=1}_J(X,\alpha,\beta)/\mathbb{R}.\]
Here, $\#$ denotes the $\mathcal{M}^{I=1}_J(X,\alpha,\beta) \subset \mathcal{M}_J(X,\alpha,\beta)$ denotes the subspace of $I = 1$ curves, where $I$ denotes the {\em embedded contact homology} index, which we call the {\em ECH index}.  We will not need the precise details of the ECH index in these notes (we refer the reader to \cite{ECHlecture} for more details) but the key points are as follows:
\begin{itemize}
 \item An $I = 1$ $J$-holomorphic current is rigid, modulo translation in the $\mathbb{R}$ direction, so that the count $\# \mathcal{M}^{I=1}_J(X,\alpha,\beta)/\mathbb{R}$ is plausible to consider.  Indeed, one can show that $\mathcal{M}^{I=1}_J(X,\alpha,\beta)/\mathbb{R}$ is a compact $0$-manifold and so consists of a finite number of points. \vspace{1 mm}
 \item An $I = 1$ $J$-holomorphic current is (mostly) embedded.  (More precisely, it has a non-trivial component, which is embedded, and then some other components consisting of covers of $\mathbb{R}$-invariant cylinders.)  This is the reason for the phrase ``embedded contact homology".
 \item The ECH index depends only on the relative homology class of the curve.  We write $H_2(Y_\phi;\alpha,\beta)$ for the space of relative homology classes from $\alpha$ to $\beta$.  
 \end{itemize}
 
 \begin{figure}[h!]
 \caption{The kind of curves counted by the PFH differential}
 \label{pfhd}
\centering
\def\svgwidth{.9\textwidth}
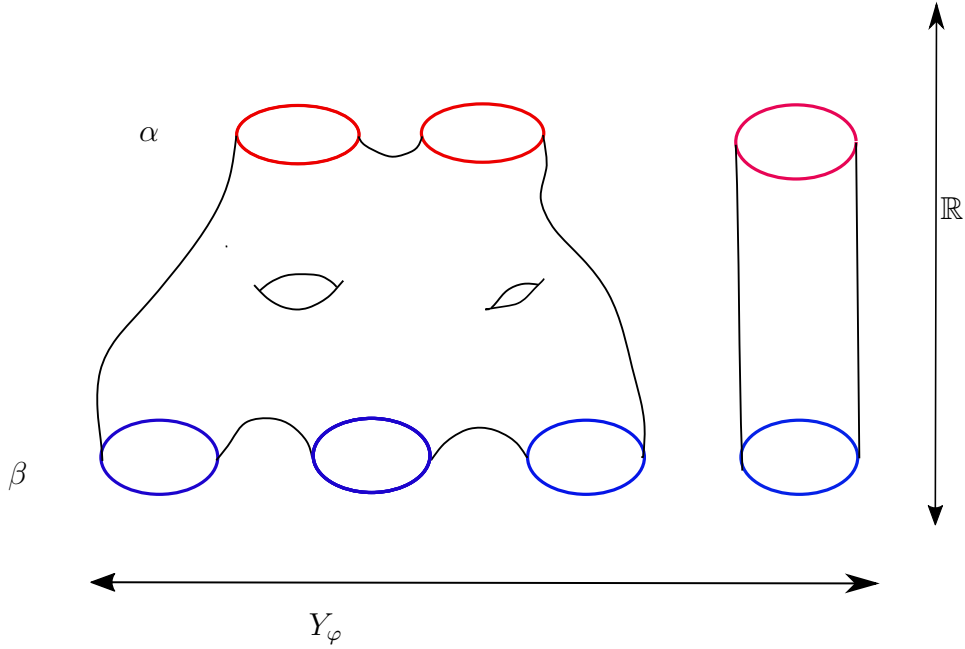
\end{figure}
 
 It is shown in \cite{HTJSG1, HTJSG2} that $\partial^2 = 0$.  Thus, the homology $PFH(S^2,\phi,d,J)$ is defined.  In principle, the homology could depend on the choice of admissible almost complex structure $J$, or on the choice of $\phi$.  However, it is shown in \cite{leetaubes} that in fact $PFH(S^2,\phi,d,J)$ recovers the Seiberg-Witten invariant of the mapping torus; as motivation, one can think of this as a kind of analogue for mapping torii of Taubes' ``Seiberg-Witten = Gromov" theorem from the '90s.  More precisely, there is a canonical isomorphism
\begin{equation}
\label{eqn:echsw}
PFH(S^2,\phi,d,J) \simeq \widehat{HM}(Y_\phi,\mathfrak{s}_d), 
\end{equation}
where $\widehat{HM}(Y_\phi,\mathfrak{s}_d)$ denotes a version of the Seiberg-Witten Floer cohomology (the ``negative monotone" version) for the mapping torus $Y_\phi$, in a particular spin-c structure $\mathfrak{s}_d$ determined by $d$.  As $\widehat{HM}$ is known to depend only on the pair $(Y_\phi,\mathfrak{s}_d)$ we obtain a well-defined invariant $PFH(S^2,d)$.  We will not need to know much about Seiberg-Witten Floer theory here beyond its formal properties and we refer the reader to \cite{KMbook} for more detail.

\subsubsection*{Twisted PFH}

For our purposes --- i.e. to define spectral invariants --- we require a slight variant of $PFH$, called the {\em twisted PFH}, which allows one to keep track of the relative homology class of the relevant pseudoholomorphic curves.  This requires a choice of {\em reference cycle} $\gamma$, namely an embedded degree $1$ loop in $Y_\phi$, together with a choice of (homotopy class of) trivialization of $V$ over $\gamma$.   Now define a {\em twisted PFH generator} of degree $d$ to be a pair $(\alpha,Z)$ where $\alpha$ is a degree $d$ PFH generator and $Z \in H_2(Y_\phi,\alpha,d\gamma)$.  We define the {\em twisted PFH chain complex} $\widetilde{PFC}(S^2,\phi,d)$ to be freely generated over $\mathbb{Z}_2$ by degree $d$ twisted PFH generators.  The differential is defined analogously to in the PFH case, except that the coefficients are given by the formula
\[ \langle \partial (\alpha,Z), (\beta,Z') \rangle = \# \lbrace C \in \mathcal{M}^{I=1}_J(X,\alpha,\beta), [C] + Z' = Z \in H_2(Y_\phi,\alpha,d \gamma) \rangle.\]
As before, by the argument in \cite{HTJSG1, HTJSG2} $\partial^2 = 0$, and as before the homology is an invariant because it agrees with the relevant Seiberg-Witten invariant\footnote{The twisting on the PFH side corresponds to quotienting out by a smaller group of gauge transformations on the Seiberg-Witten side, namely those in the component of the identity.} of the mapping torus.  We denote the homology $\widetilde{PFH}(S^2,d,\gamma).$  The twisted PFH has a $\mathbb{Z}$-{\em grading} $gr$ induced by the grading on twisted PFH generators given by
\[ gr(\alpha,Z) = I(Z)\]
and we let $\widetilde{PFH}_*(S^2,d,\gamma)$ denote the grading = * piece.   

In the case of $S^2$, one can compute that: 
\begin{equation}
\label{eqn:grad}
\widetilde{PFH}_*(S^2,d,\gamma) = \mathbb{Z}_2, \quad * - d = 0 \hspace{2 mm} mod \hspace{2 mm} 2,
\end{equation}
\[ \widetilde{PFH}_*(S^2,d,\gamma) = 0, \quad * - d = 1 \hspace{2 mm} mod \hspace{2 mm} 2.\]
To compute this, one uses the fact that it is an invariant, and takes as $\phi$ an irrational rotation.  Then, an index calculation shows that all curves would have to have even ECH index, thus the differential vanishes, and one just needs to compute the ECH index, which is topological.


\subsection{PFH spectral invariants}
\label{sec:pfhspec}

The twisted PFH chain complex has a canonical filtration, and we can use this to construct the promised spectral invariants.  We now explain this.  We continue with the setup from the previous section.  A twisted PFH generator $(\alpha,Z)$ has an {\em action} defined by
\[ \mathcal{A}(\alpha,Z) = \int_Z \omega_\varphi.\]
We now let $\widetilde{PFC}^L(S^2,d,\varphi,\gamma)$ denote the subspace generated by twisted PFH generators of action $\le L$.  By the admissibility condition on $J$, the two-form $\omega_\varphi$ is pointwise nonnegative on any $J$-holomorphic curve $C$, and so $\widetilde{PFC}^L(S^2,d,\varphi,\gamma)$ is in fact a subcomplex.  We let  $\widetilde{PFH}^L(S^2,d,\varphi,\gamma)$ denote its homology.  There is an inclusion induced map
\begin{equation}
\label{eqn:incind}
\widetilde{PFH}^L(S^2,d,\varphi,\gamma) \to \widetilde{PFH}(S^2,d,\varphi,\gamma).
\end{equation}
Now let $\sigma$ be any nonzero class in $\widetilde{PFH}(S^2,d,\varphi,\gamma)$ and let $\varphi$ be nondegenerate.  We define the corresponding {\em PFH spectral invariant} $c_\sigma(\phi)$ to be the minimum, over $L$, such that $\sigma$ is in the image of the map.  We can think of this as a minmax: any given representative of $\sigma$ has a certain amount of action required to represent it, and we take the action of the most efficient representative.  This definition then extends to degenerate $\varphi$ by approximating by nondegenerate $\varphi_n$ and taking the limit of $c_\sigma(\varphi_n)$; it follows from either the Hofer Lipschitz or Continuity Axioms that this does not depend on the choice of approximating $\varphi_n$.





We can now explain the definition of the PFH spectral invariants for a map 
\[ \phi \in \Diffeo_c(D^2,dx\wedge dy).\]  
First, we embed $(D^2,dx\wedge dy)$ as the northern hemisphere in $(S^2,\omega_{std})$.  Then, any $\phi \in \Diffeo_c(D^2,dx\wedge dy)$ extends via the identity to an area-preserving diffeomorphism of $S^2$.  Next, we choose as our reference cycle $\gamma$ the simple Reeb orbit corresponding to the south pole $p_-$; this inherits a homotopy class of trivializations from a trivialization of $T_{p_+} S^2$.  
We next want to find some nonzero classes $\sigma$.  We use the calculation \eqref{eqn:grad}: it is convenient to take $\sigma_d$ to be the unique class in grading $d$.  Finally, at long last we define the promised {\em PFH spectral invariant}
\[ c_d(\phi) = c_{\sigma_d}(\varphi).\]
In other words, $c_d(\phi)$ is the PFH spectral invariant associated to the unique class of degree $d$ and grading $d$.  

\subsection{The PFH spectral invariant axioms}

We will now say a bit about the proofs of the PFH spectral invariant axioms.  Full proofs are  beyond the scope of these notes, but we will try to give the general idea behind the arguments.


\subsubsection*{The Identity Axiom}

This follows without much difficulty by direct computation: one just approximates the identity by irrational rotations of the two-sphere, for which it is not hard to compute the spectral invariants.

\subsubsection*{The Monotonicity and Hofer Lipschitz Axioms}

The proofs of these axioms are similar.  It suffices to prove them when $\psi^1_H$ and $\psi^1_K$ are nondegenerate.  

To prove them, one uses the fact that $\widetilde{PFH}$ is an invariant.   More precisely, by invariance, there is an  isomorphism 
\[ \Psi_{H,K}: \widetilde{PFH}(\varphi^1_H,\gamma) \simeq \widetilde{PFH}(\varphi^1_K,\gamma).\]
One might hope that this isomorphism is induced by counting pseudoholomorphic curves, similarly to the definition of the chain complex differential $\partial$.  For example, one can form a symplectic manifold $X = \mathbb{R}_s \times S^1_t \times S^2$, with an almost complex structure $J$, such that for $s$ sufficiently large $(X,J)$ agrees with the symplectization of $Y_{\varphi^1_H}$, and for $s$ sufficiently negative $(X,J)$ agrees with the symplectization of $Y_{\varphi^1_K}$.  One might then hope that one can define $\Psi_{H,K}$ by defining a chain map counting $I = 0$ $J$-holomorphic curves in $X$.   

If this were true --- i.e. that there is a chain map inducing the isomorphism $\Psi_{H,K}$, defined by counting $I = 0$ curves --- then one could prove the Hofer Lipschitz property similarly to the argument in \S\ref{sec:pfhspec} that the chain complex differential does not increase action.  Namely, if one assumes this for the moment then, one finds, after a little bit of algebraic manipulation (see \cite{simp}), that
\[ c_{\sigma}(\varphi^1_H) - c_\sigma(\varphi^1_K) \ge \int_C \omega_X + \int_C G' ds \wedge dt,\]
for some $J$-holomorphic curve $C$.  Here, $\omega_X$ is a certain two-form that we do not write out for brevity (see \cite{simp} for the formula), but that it is manifestly pointwise nonnegative along $C$ due to the fact that $C$ is $J$-holomorphic.  The function $G$ interpolates between $H$ and $K$: we have $G = K + \beta(s) (H - K)$, where $\beta:\mathbb{R} \to [0,1]$ is some function which is $0$ for $s$ sufficiently negative and $1$ for $s$ sufficiently positive.  Thus, by pointwise nonnegativity 
\[ c_{\sigma}(\varphi^1_H) - c_\sigma(\varphi^1_K) \ge \int_C G' ds \wedge dt.\]
Now if $H \ge K$, then $G'$ is nonnegative, and so we obtain the Monotonicity Axiom.  As for the Hofer Lipschitz property, the crux of what remains to be shown is the bound 
\[ |\int_C G' ds \wedge dt| \le d ||H-K||_{1,\infty},\]
and the key step for this is to estimate the integral by projecting to the $(s,t)$ co-ordinates; this projection has degree $d$.

In the above argument, we assumed that the map $\Psi_{H,K}$ comes from counting $I = 0$ curves.  For technical reasons, it is very difficult to try to actually define $\Psi_{H,K}$ that way.  Instead, one uses the isomorphism with Seiberg-Witten theory as a workaround.  Namely,  recall that $\widetilde{PFH}$ is isomorphic to a version $\widehat{HM}$ of the Seiberg-Witten Floer cohomology.   It is known by work of Kronheimer-Mrowka \cite{KMbook} that the cobordism $X$ induces a map on $\widehat{HM}$.  Thus, the cobordism $X$ induces a map $\Psi_{H,K}$ on PFH.  As defined this way, the map counts Seiberg-Witten solutions; in the above argument, one wants to use pseudoholomorphic curves.  However, it is shown in \cite{chen}, building on ideas of \cite{cc2}, that the map induced on PFH does satisfies a {\em holomorphic curve axiom}, asserting that if $\langle \Psi_{H,K} \alpha, \beta \rangle \ne 0$, there is some (possibly broken) ECH index $0$ $J$-holomorphic curve from $\alpha$ to $\beta$; one can regard this as a variant of Taubes' Seiberg-Witten to Gromov theorem, for the manifold $X$.  This broken curve suffices for the arguments in the previous paragraph. 

\subsubsection*{$C^0$-continuity}

The starting point for the proof of the $C^0$-continuity Axiom is that the spectral invariants $c_d(\varphi)$ are not just arbitrary real numbers: they take values in the {\em action spectrum} 
\[ \mathcal{A}_d(\varphi) := \lbrace \mathcal{A}(\alpha,Z) \rbrace \subset \mathbb{R},\]
where $\alpha$ is an orbit set.  Importantly, 
$\mathcal{A}_d(\varphi)$ is a measure zero subset; the idea for proving this is to use Sard's theorem.
The next key point is that the proof uses only the other axioms of the PFH spectral invariants, rather than any particular details about PFH.  In particular, the Hofer-Lipschitz property, which implies the bound
\begin{equation}
\label{eqn:hlp}
|c_d(\psi) - c_d(\varphi) | \le d ||\psi^{-1} \circ \varphi ||
\end{equation}
plays a crucial role.

For brevity, let us now restrict to proving $C^0$-continuity at the identity.  The general case proceeds by broadly similar ideas and is not too much more involved.  In fact, it is expected that there is a ``quantum product" structure on PFH, see \cite{hs}, and
via this
one should be able to reduce the general case to continuity at the identity.

Thus, we are given $\epsilon > 0$ and we want to find a $\delta > 0$ such that if $d_{C_0}(\varphi,id) < \delta$, then $c_d(\varphi) < \epsilon$.

To get the desired bound, we use the following trick, which is a variant of an idea from \cite{sey}.  Let us assume for a moment that we can find some auxiliary diffeomorphism $f$ such that
\begin{equation}
\label{eqn:trick}
c_d(\varphi \circ f) = c_d(f).
\end{equation}
Then, by the Hofer Lipschitz property \eqref{eqn:hlp}, and the Identity Axiom we get on the one hand that
\begin{equation}
\label{eqn:int}
|c_d(\varphi \circ f) | = |c_d(f)| = |c_d(f) - c_d(id)| \le d ||f||.
\end{equation}
On the other hand, by again applying \eqref{eqn:hlp}, we get that
\[ |c_d(\varphi)| \le | c_d(\varphi \circ f) | + d ||f|| ,\]
and so combining this inequality with \eqref{eqn:int} gives the crucial bound
\[ |c_d(\varphi)| \le 2d ||f||.\]

Thus, the problem is transformed to finding an $f$, with $||f||$ small, such that left composition with $\varphi$ does not change $c_d$ when $d_{C^0}(\varphi,id) < \delta$.  It is here that the fact that the $c_d$ take values in a measure zero set is crucial.  

To elaborate, let us explain how to find such an $f$.  Recall that in our construction, any $\varphi$ is supported in the northern hemisphere $S^+$.  We now take $F$ to be a function on $S^2$ with all critical points in the southern hemisphere, such that the associated Hamiltonian diffeomorphism $\psi^1_F$ has the property that its only periodic points of period $\le d$ are the critical points of $f$, and such that
\[ max(F) - min(F) < \frac{\epsilon}{2d}.\]
It is easy to construct such a function $F$; for example a $C^2$-small function with all critical points in the southern hemisphere suffices.  We can assume in addition that the support of $F$ is disjoint from the south pole $p_-$.  We then take $f = \psi^1_F$.

It is immediate from the construction that $2d ||f|| < \epsilon$, and so it remains to explain why \eqref{eqn:trick} holds.  The point is that one calculates that $f$ and $\varphi \circ f$ have the same periodic points with period $\le d$ when $\varphi$ is $C^0$-close to the identity; one also computes that their actions are also the same for $f$ or $\varphi \circ f$.  Thus
\[ \mathcal{A}_d(f) = \mathcal{A}_d(\varphi \circ f).\]
Now write $\varphi = \psi^1_K$ for some (possibly time-dependent Hamiltonian) $K$.  With a little more work, a similar argument shows that in fact
\[ \mathcal{A}_d(f) = \mathcal{A}_d(\psi^s_K \circ f)\]
for all $0 \le s \le 1$
Thus, as $s$ varies, $c_d(\psi^s_K \circ f)$ is taking values in a measure zero set that is independent of $s$; on the other hand, it follows from the Hofer-Lipschitz property that this is a continuous function of $s$.  Thus, it must be constant which establishes \eqref{eqn:trick}.

\begin{remark}
\label{rmk:closing}
In proving the closing lemma Theorem~\ref{thm:closing}, given the Weyl law, we use a trick pioneered by Irie \cite{irie} that is similar to the idea at the end of the above proof.  Namely, for that proof, by a Baire category theorem argument, one can reduce to proving the existence of a periodic point in a fixed open set after small perturbation.  So, given an open set $U$, one considers a $C^{\infty}$-small perturbation, supported in $U$, that adds Calabi invariant, with the goal of producing a periodic point in $U$.  The point is that if no new periodic orbits appear in $U$, the action spectrum stays the same, which, by arguments similar to the ones above, shows that the spectral invariants must stay the same, contradicting the Weyl law.
\end{remark}

\subsubsection*{The Weyl Law}

There are several ways to prove this.  

\vspace{1 mm}


To prove Theorem~\ref{thm:simpconjecture}, it is clear from the proof we gave in \S\ref{sec:proof} that one only actually needs to know the Weyl Law for smooth monotone twist maps.  
Recall that (taking polar co-ordinates of the disc), there are smooth maps that are the identity near the boundary
\[  T(r,\theta) \to (r, \theta+ f(r))\] 
where $f:[0,1] \to \mathbb{R}$ is a concave non-increasing function.  In this special case, one can actually prove the Weyl Law by direct computation.  That is, the (twisted) PFH chain complex of such a smooth monotone twist $\phi$ can be computed directly.  This is the approach taken in the first proof of Theorem~\ref{thm:simpconjecture} in \cite{simp}.

To say a bit more about this, recall that the generators are certain sets of periodic orbits.  In the case of monotone twists, which preserve the radial co-ordinate, it is easy to work these out explicitly: for any $\lbrace r = constant \rbrace$ circle twisted by an irrational number, there are no orbits, and when the twist is by a rational number $p/q$, there is entire circle of periodic orbits.  Since the orbits in this case come in families, the map $\phi$ can not be nondegenerate; however, it can be regarded as ``{\em Morse-Bott}" nondegenerate, and this is a much studied situation.
In particular, standard theory in this setup says that the map $\phi$ can be perturbed to a map $\phi_d$, such that all orbits of period $\le d$ are non-degenerate, and such that each Morse-Bott circle of orbits corresponding to a rational twist by $p/q$ split into an elliptic orbit $e_{p/q}$ and a hyperbolic orbit $h_{p/q}$.  (In the Morse-Bott picture, the idea is that since we have an $S^1$-family of orbits, we can take a perfect Morse function on $S^1$ and perturb the map using the data of this function  -- then the two critical points correspond to the two new orbits.)  Thus we obtain satisfactory non-degenerate perturbations of $\phi$ whose orbits can be catalogued.  One can similarly compute the action and the grading of any orbit set.  

To understand the chain complex, one has to in addition understand all $I = 1$ $J$-holomorphic curves, for a suitable $J$.  This is more complicated.  Thankfully, Taubes has previously catalogued all the relevant $J$-holomorphic curves in a related situation \cite{beasts}, and some inductive arguments allow a reduction to Taubes' computation; in practice, this goes via relating to work of Choi \cite{choi}.

To put all of this together, one can build a combinatorial model of the filtered PFH chain complex, where twisted PFH generators correspond to certain ``lattice paths" in the plane $\mathbb{R}^2$, taking vertices on the standard integer lattice, and the differential is described by a ``rounding corner" operation, see Figure~\ref{fig:rounding}.
This is inspired by, and closely related to, work by Hutchings-Sullivan building similar models, see e.g. \cite{hs}. From here, some elementary algebra and combinatorics reduces the Weyl law to the isoperimetric inequality for possibly nonstandard norms.

If one wants a more general Weyl law, one way is to use Seiberg-Witten invariants.   We give an account in \S\ref{sec:weyl}.
Alternatively, one can use a ``ball-packing argument", reducing the computation to  a similar computation for the boundary of the standard symplectic ball; see \cite{eh}.

 \begin{figure}
  \centering
  \includegraphics[width=4cm]{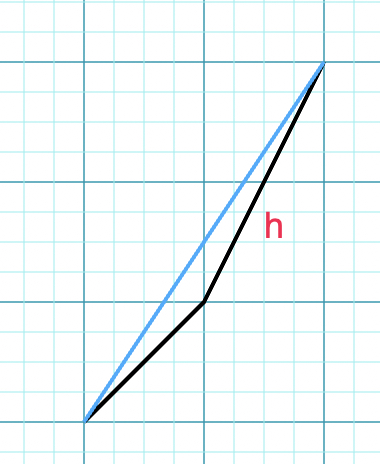}
  \def\svgwidth{.5 \textwidth}
  \caption{An example of the corner rounding operation.  In the combinatorial model, this represents a differential from an orbit set (in this case, consisting of two orbits) to another.}
  \label{fig:rounding}
\end{figure}

\section{Two or infinity}

\subsection{The main theorem}

We now turn our attention to Conjecture~\ref{conj:hwz}.  We will try to illustrate the ideas behind the proof of the following theorem.    Let $\lambda$ be a contact form and recall the contact structure $\xi := Ker(\lambda)$.  
It has a well-defined first Chern class $c_1(\xi) \in H^2(Y;\mathbb{Z}).$  

\begin{theorem} \cite{CGHHL2}
\label{thm:hwz}
Let $Y$ be a closed connected three-manifold.  The Reeb vector field associated to any contact form $\lambda$ on $Y$ has two or infinitely many simple periodic orbits as long as $c_1(\xi) \in H^2(Y,\mathbb{Z})$ is torsion.
\end{theorem}  

This resolves Conjecture~\ref{conj:hwz} (Hofer-Wysocki-Zehnder's two or infinity conjecture) in the affirmative, since if $Y$ is star-shaped, we must have $Y = S^3$, and then $H^2(S^3; \mathbb{Z}) = 0$.  At the same time, it applies to much more general flows.  For example, it holds for any Reeb flow on any rational homology sphere.  

Our goal in this section is to explain the key ideas behind its proof.

\begin{remark}
In the case where there are two orbits, much can be said.  For example, it was shown in \cite{HT}
that 
if a contact form on a closed connected three-manifold $Y$ has exactly two simple Reeb orbits, then $Y$ is a lens space and the orbits are the core circles of a genus $1$ Heegaard decomposition.
Later, this was extended to all contact forms 
 in \cite{CGHHL},
 with the additional conclusion 
 that in this case
 the contact structure is standard.  (No condition on $c_1(\xi)$ is required for these results.)  Thus, Theorem~\ref{thm:hwz} applies in many situations to guarantee the existence of infinitely many orbits.  
\end{remark}

\subsection{Finsler metrics: an application}

Theorem~\ref{thm:hwz} also has implications for Finsler geometry.  A {\em Finsler metric} is a generalization of a Riemannian metric, where there is a (not necessarily symmetric) norm\footnote{Another way to say this is that we require a smooth family of convex subsets of the tangent bundle, and then the Finsler metric of a tangent vector is the maximum scaling putting it in the corresponding convex subset.} on each tangent space, but the norm is not required to come from an inner-product.  One can define geodesics in this situation, analogously as in the Riemannian case.  A theme of research in the area is to understand what facts about Riemannian geodesic flows continue to hold in the Finsler setting; for more about this, see the perspective in the ICM address of Anosov \cite{anosov}.  A notable example that is {\em unlike} the Riemannian case is the study of closed geodesics on $S^2$.  In the Riemannian case, it is known \cite{Franks, Bangert2} that there are always infinitely many prime closed geodesics.  However, Katok constructed in \cite{Katok} a Finsler metric with exactly two closed geodesics.  The following had been a longstanding \cite{paiva, BurnsMatveev, long}
conjecture related to Conjecture~\ref{conj:hwz}:

\begin{conjecture}
\label{conj:fins}
Every Finsler metric on $S^2$ has either two or infinitely many prime closed geodesics.
\end{conjecture}      

Theorem~\ref{thm:hwz} resolves the above conjecture in the affirmative.  The reason is that, along the lines that we previously explained, 
there is a contact form on the unit tangent bundle of any Finsler manifold, such that the Reeb flow lines project to geodesics; one can then deduce the resolution of Conjecture~\ref{conj:fins} by applying Theorem~\ref{thm:hwz} to $\mathbb{R}P^3$.

\subsection{Outline of the argument}
 \label{sec:outline}
 
 Let us now outline the proof of Theorem~\ref{thm:hwz} before getting into the details.
  The basic idea of the proof is to find a {\em global surface of section} (GSS) $S$ for the Reeb flow.  This is a compact surface $S$, immersed in $Y$ such that the interior $int(S)$ is embedded and transverse to the Reeb vector field, the boundary of $S$ is on Reeb orbits, and every flow line hits $S$ both forwards and backwards in time.  Given such an $S$, we can define a {\em first return map}
  \[ P: int(S) \to int(S),\]
 by taking a point to the next place at which the flow line through the point hits $int(S)$.  Then, the periodic points of $P$ are in bijection with the Reeb orbits, and so we have achieved a kind of dimensional reduction.  Moreover, it follows from the definitions that $d\lambda$ restricts to $S$ as an area form, preserved by $P$.
 
 We will be interested in the case where $S$ is an annulus, with $\partial S$ on distinct Reeb orbits.  If one can find such an $S$, then Theorem~\ref{thm:hwz} would follow from the following celebrated result of Franks:
 
 \begin{theorem}\cite{Franks}
 Any area-preserving homeomorphism of an open annulus with at least one periodic point has infinitely many.
 \end{theorem}

Thus, our problem is now to find the desired global surface of section.  Following ideas of \cite{convex}, we will try to find this GSS by projecting a $J$-holomorphic curve from $\mathbb{R} \times Y$ to $Y$.   To prove Theorem~\ref{thm:hwz} we can assume that there are finitely many Reeb orbits, and this simplifies matters.  To find the right curve, we will use a cousin of the aforementioned periodic Floer homology (PFH).   To implement this, there are essentially two main challenges we need to overcome:
\begin{itemize}
\item Like its cousin PFH, ECH is only defined for nondegenerate contact forms.  However, Theorem~\ref{thm:hwz} has no nondegeneracy hypotheses. 
\item Like its cousin PFH, the relevant $J$-holomorphic curves in ECH can have genus and many punctures.  To get an annular global surfaced of section, one would like to find cylindrical $J$-holomorphic curves.
\end{itemize}

To get around the first issue, the idea is to approximate $\lambda$ with nondegenerate contact forms $\lambda_n$.  Then, ECH is defined for the $\lambda_n$ and so we can bring it to bear on the situation.   However, the approximation process immediately gives rise to several new issues.  First of all, although we could assume that $\lambda$ had finitely many simple Reeb orbits, we can no longer assume this for $\lambda_n$.  Certainly, there exist contact forms with no annular global surface of section (there are for example topological obstructions to this), so one has to clarify what kind of additional structure the $\lambda_n$ have that can be used.  The next issue is that even if one can show that the $\lambda_n$ have annular global surfaces of section, it is not at all clear that $\lambda$ itself inherits a GSS in the limit.  

As for the second bullet point, one needs to develop new tools for studying the topology of ECH/PFH curves.  We introduce a ``score" built for this purpose, and this on top of some other new ideas, give the result; it is here that the assumption that $c_1(\xi)$ is torsion is used.

The above represent the major challenges that need to be overcome for this strategy and in this section we will try to give a sense for how all of this works.

\begin{remark}
One might wonder if non-annular global surfaces of section suffice for our purposes.  Certainly, non-annular curves are easier to find with ECH so one might hope that it is easier to find a non-annular GSS.  If we had a non-annular GSS, then one could still deduce the existence of infinitely many simple Reeb orbits: in the genus zero case, one could reduce to Franks' theorem above, and in the higher genus case, one can use the fact that the return map in addition must have zero flux (when its flux is defined).  However, as we will see in our arguments, it seems very hard to show that non-annular $J$-holomorphic curves for the $\lambda_n$ project to global surfaces of section under our hypotheses; and, moreover, even if such curves did exist, we will see in our arguments that they do not seem well-suited to taking limits.  This is why we are looking for annular global surfaces of section.
\end{remark}

 \begin{figure}
  \centering
  \includegraphics[width=6cm]{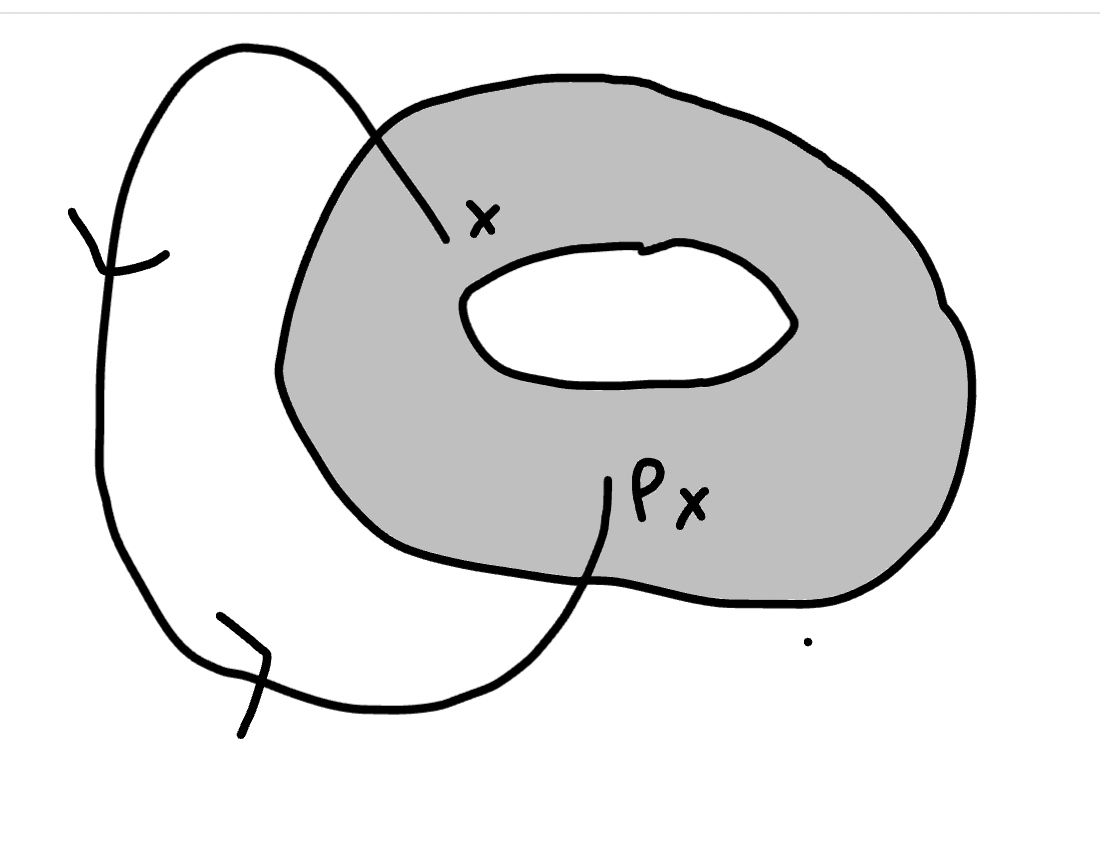}
  \def\svgwidth{,5 \textwidth}
  \caption{An annuluar global surface of section.  A typical flow line hitting the surface is shown.  The map $P$ is the first return map.}
  \label{fig:gss}
\end{figure}

\subsection{Embedded contact homology}
\label{sec:echdef}
We begin by introducing ECH and summarizing the key properties of ECH that we will need.  For our purposes here, we want to use ECH to produce many $J$-holomorphic curves.  Much of this is quite analogous to our discussion of PFH in \S\ref{sec:pfh}.  The definition of ECH is also due to Hutchings.

\subsubsection{Definition of ECH}  
Let $\lambda$ be a nondegenerate contact form on a three-manifold $Y$.   We define orbit sets and ECH generators just as in the PFH case: an ECH generator is a finite set $\lbrace (\alpha_i,m_i) \rbrace$ where the $\alpha_i$ are distinct embedded Reeb orbits, the $m_i$ are positive integers, and $m_i = 1$ if $\alpha_i$ is hyperbolic.  We also define the differential just as in the PFH case: we have
\[ \partial \alpha = \sum_\beta \langle \alpha, \beta \rangle \beta,\]
where $\langle \alpha, \beta \rangle$ counts $I = 1$ $J$-holomorphic currents in the symplectization $X = \mathbb{R} \times Y$, modulo translation, asymptotic to $\alpha$ at $+\infty$ and asymptotic to $\beta$ at $-\infty$, for generic $J$; just as in the PFH case we are assuming that $J$ is admissible, in the sense that it maps $\partial_s$ to the Reeb vector field and preserves $\xi$ compatibly with $d\lambda$; we will assume all almost complex structures are admissible here and below without further comment.  As we explained in \S\ref{sec:pfh}, such curves are mostly embedded, which is why the theory is called embedded contact homology.  (Specifically, an $I = 1$ curve has a unique nontrivial component which is embedded, and all the other components are covers of $\mathbb{R}$-invariant cylinders.)  It is shown in \cite{HTJSG1,HTJSG2} that $\partial^2 = 0$ and so the homology is well-defined.   It has a (relative) grading induced by the ECH index.  If we fix a class  $\Gamma \in H_1(Y;\mathbb{Z})$, then there is a well-defined subcomplex $ECH(Y,\lambda,\Gamma)$ consisting of ECH generators, in class $\Gamma$.  Just as in the PFH case, there is a canonical isomorphism 
\begin{equation}
\label{eqn:echsw}
ECH(Y,\lambda,\Gamma)_* \simeq \widehat{HM}^{-*}(Y,s_\Gamma),
\end{equation}
where $\widehat{HM}$ is the Seiberg-Witten Floer cohomology defined by Kronheimer-Mrowka in \cite{KMbook}, and $s_\Gamma$ is a certain spin-c structure determined by $\Gamma$ and the contact structure $\xi$.   The isomorphism implies that $ECH$ is independent of the contact form $\lambda$; we will therefore sometimes write it as $ECH(Y,\lambda,\xi)$.  It also allows us to convert topological information about the manifold $Y$, encoded by the Seiberg-Witten invariants, into symplectic data.  Here is an illustrative example showcasing the power of this:
\begin{example}
Recall from \S\ref{sec:intro} the ``Weinstein conjecture", stating that a Reeb vector field on a closed manifold always has a Reeb orbit.  The isomorphism \eqref{eqn:echsw} implies this for three-manifolds, given known facts about $\widehat{HM}$.  The reason is that it is known by \cite{KMbook} that $\widehat{HM}$ always has infinite rank, when the first Chern class  $c_1(s_\Gamma)$ is torsion.  In fact, currently, the only known proofs of the three-dimensional Weinstein conjecture use Seiberg-Witten theory.
\end{example}

\subsubsection{Additional structures on ECH}
\label{sec:add}

We will be interested in an additional structure on ECH, called the ``$U$-map".  This counts $I = 2$ $J$-holomorphic curves, through a marked point, for generic choice of $J$.    This decreases the ECH grading by $2$.  Just as with the differential curves, an ECH $U$-map curve $C$ decomposes into components
\begin{equation}
\label{eqn:ucurves}
C = C_0 \cup C_1
\end{equation}
where $C_1$ is a single component, which is embedded, and $C_0$ consists of covers of $\mathbb{R}$-invariant cylinders; we call such cylinders {\em trivial cylinders}.
Importantly for our purposes, the component $C_1$ has Fredholm index two, implying that it lives in a two-dimensional moduli space.  The basic idea is that we hope that the projection of $C_1$ will be a GSS, and that this can be detected by the curves in the same moduli space component of $C_1$ foliating the symplectization $X$.

The isomorphism \eqref{eqn:echsw} guarantees us an abundance of $U$-map curves.  Indeed, there is a corresponding $U$-map on $\widehat{HM}$, and it is shown in \cite{KMbook} that when $c_1(s_\Gamma)$ is torsion, the $U$-map is an isomorphism in sufficiently negative grading.  One computes that $c_1(s_\Gamma) = c_1(\xi) + 2 \text{PD}(\Gamma)$, so by \eqref{eqn:echsw}, the $U$-map on $ECH(Y,\lambda,\Gamma)$ is an isomorphism in high enough grading whenever $c_1(\xi) + 2 \text{PD}(\Gamma)$ is torsion.  We will always assume from here on that we have chosen $\Gamma$ to satisfy this.  In the context of Theorem~\ref{thm:hwz}, we are assuming that $c_1(\xi)$ is itself torsion, so  we can take $\Gamma = 0$.

Just like with PFH, we can define spectral invariants with ECH.  In fact, it is a little bit simpler with ECH, since no reference cycle $\gamma$ is needed and we do not need to twist the theory.  An orbit set $\alpha = \lbrace (\alpha_i,m_i) \rbrace$ has an {\em action}
\[ \mathcal{A}(\alpha_i,m_i) = \sum m_i \int_{\alpha_i} \lambda.\]
For admissible $J$, the differential $\partial$ decreases the action by the same argument as in the PFH case (see \S\ref{sec:pfhspec}) and so the subspace $ECH^L$ generated by orbit sets of action $\le L$ is a subcomplex.  Given a nonzero class $\sigma \in ECH(Y,\xi,\Gamma)$, we then define $c_{\sigma}(\lambda)$ to be the infimum, over $L$, such that $\sigma$ is in the image of the inclusion induced map
\[ ECH^L(Y,\lambda,\Gamma) \to ECH(Y,\xi,\Gamma).\]
Thus, more concretely, to each representative of $\sigma$, we look at the maximum of the actions of each orbit set in the representatives, and then we minimize this quantity over all possible representatives.
We call this the {\em ECH spectral invariant} associated to $\sigma$.  Just as in the PFH case, it enjoys various favorable continuity properties.  In particular, it is $C^0$-continuous with respect to the contact form.

\subsubsection{The ECH Weyl law}

Similarly to the PFH case, there is also an ECH Weyl law.  It in fact predates the PFH one.  To state it, recall that $ECH(Y,\lambda,\Gamma)$ has a relative grading induced by the ECH index.  In general, this takes values in\footnote{The reason for the appearance of $d$ comes from the ``index ambiguity formula": one would like to define the relative grading between two ECH generators to be the ECH index of a relative homology class between them, but one can compute that two such relative homology classes $Z$ and $Z'$ satisfy $I(Z) - I(Z') = \langle c_1(\xi) + 2 PD(\Gamma), Z - Z' \rangle$, see \cite{ECHlecture}. } $\mathbb{Z}/d$, where $d$ is the divisibility of $c_1(\xi) + 2 \text{PD}(\Gamma)$ in $H^2$ modulo torsion.  In our case, $d = 0$ so we have a relative $\mathbb{Z}$-grading.  To state the Weyl law, let us refine this to a (non-canonical) absolute $\mathbb{Z}$-grading $gr$ by arbitrarily fixing a class to have grading $0$.  We can now state the Weyl law:

\begin{theorem} \cite{ECHasymptotics}
\label{thm:weyl}
Let $(Y,\lambda,\Gamma)$ be a closed three-manifold such that $c_1(\xi) + 2 PD(\Gamma)$ is torsion.  Then
\[ \lim_{k \to \infty} \frac{c^2_{\sigma_k}(\lambda)}{gr(\sigma_k)} = \int_Y \lambda \wedge d \lambda\]
whenever $\lbrace \sigma_k \rbrace$ is any sequence of nonzero classes with definite gradings tending to $+\infty$.
\end{theorem}

An account of the proof of this is given in \S\ref{sec:weyl}.

In the PFH case, what was important for our purposes was the recovering of the Calabi invariant through Floer-theoretic considerations.  For our purposes in proving Theorem~\ref{thm:weyl}, what will be important is instead the {\em growth rate} in the Weyl law, stating that the action grows like the square root of the grading.  We will see in \S\ref{sec:utowers} that this gives rise to many ``low action" curves.   We should note, though, that there are other applications where the appearance of the volume is essential.  Probably the most spectacular is the aforementioned work of Irie and Asaoka-Irie on the closing lemma \cite{irie, ai}; for another application, see \cite{CGHHL}.  

\subsection{$U$-map curves and the $J_0$ index}
\label{sec:ujo}

To understand the topology of the $U$-map curves, there is a variant of the ECH index called the $J_0$ index which is useful.  For brevity, we will not state the definition of $J_0$, but we will just give the formula for the difference, since this will be important for our purposes.  Namely, we have:
\begin{equation}
\label{eqn:diff}
(I - J_0)(Z) = 2c_{\tau}(Z) + CZ^{top}_{\tau}(Z),
\end{equation}
where $Z \in H_2(Y,\alpha,\beta)$ is any relative homology class between orbit sets $\alpha$ and $\beta$.  Here,  $\tau$ denotes a trivialization of $\xi$ over all Reeb orbits and $c_{\tau}$ denotes the {\em relative first Chern class} of $\xi$ restricted to $Z$: this is a signed count of zeros of a generic section of $\xi$, extending the trivialization $\tau$.  Meanwhile, $CZ^{top}$ is a {\em Conley-Zehnder index} term that depends only on $\partial Z$.  To briefly explain this, let us first say a few words about the Conley-Zehnder index of a simple Reeb orbit $\gamma$.  After choosing a base point $p \in \gamma$, the trivialization $\tau$ allows us to represent the linearized flow as one goes around $\gamma$, restricted to the contact structure, as a path of symplectic matrices.   One can then extract a {\em rotation number} $\theta$ from this path (see \cite{ECHlecture} for the details) and we define 
\[ CZ_{\tau}(\gamma^m) = \lfloor m\theta \rfloor + \lceil m\theta \rceil.\]
Next, if $\alpha = \lbrace (\alpha_i, m_i) \rbrace$ is an orbit set, we define 
\[ CZ^{top}_{\tau}(\alpha) = \sum_i \CZ_{\tau}(\alpha_i^{m_i}).\]
Finally, we define
\[ CZ^{top}_{\tau}(Z) = CZ^{top}_\tau (Z) = CZ_{\tau}^{top}(\alpha) - CZ_{\tau}^{top}(\beta),\]
for $Z \in H_2(Y,\alpha,\beta).$

The following result links the $J_0$ index to the topology of the $U$-map curves.  Recall from \S\ref{sec:add} the decomposition $C = C_0 \cup C_1$ for ECH $U$-map curves.

\begin{prop} \cite{CGHHL2, HT, ir}
\label{prop:j0bound}
Let $C = C_0 \cup C_1$ be a curve counted by the ECH $U$-map.
We have 
\[ J_0([C]) = -2 + 2g(C_1) + e(C).\]
\end{prop}

In the above formula, $g(C_1)$ denotes the genus of $C_1$.  Meanwhile, the term $e(C)$, which is nonnegative, is a certain measurement of the ends of $C$.  Namely, for each orbit $\gamma_+$ at which $C_1$ has positive ends, define $e(\gamma_+)$ to be twice the number of ends at $\gamma_+$, minus $1$ if $C_0$ does not have any ends at $\gamma_+$.   We define $e(\gamma_-)$ analogously.  We then define $e(C)$ to be the sum, over all orbits $\gamma_+$ at which $C_1$ has positive ends, of $e(\gamma_+)$, plus the sum, over all orbits $\gamma_-$ at which $C_1$ has negative ends, of $e(\gamma_-)$.

In view of Proposition~\ref{prop:j0bound}, we would like to bound $J_0([C)])$.  Because $C$ is a $U$-map curve, we have that $I(C) = 2$.  The first principle that we will want to make precise is that 
\[ I \sim J_0.\]
To do this, we start with \eqref{eqn:diff}.  In our case, $c_1(\xi)$ is torsion, and we can conclude from this that we can essentially disregard the $c_\tau$ term on the right hand side of \eqref{eqn:diff}.  For example, when $c_1(\xi) = 0$ (which we can assume if we just want to prove Conjecture~\ref{conj:hwz}), we can choose a global trivialization $\tau$ of $\xi$, and then $c_{\tau}(Z) = 0$ by definition, for any $Z$.  A little more work is required in the case where $c_1(\xi)$ is torsion, but not zero, and we will return to this in a moment.  As for the $CZ^{top}$ term, the key observation is that for any orbit, $CZ_{\tau}(\gamma^m)$  is proportional to its length.  We will now explain a nice trick for getting a good bound from this.

\subsection{$U$-towers and first considerations}
\label{sec:utowers}
Recall that we are trying to make precise, and prove, the principle that $I \sim J_0$.  

To extract a useful estimate from this observation, we want to use the Weyl law Theorem~\ref{thm:weyl}.  To bring this to bear, we will want to consider the map $U^N$ for a large positive power $N$.  By the discussion in \S\ref{sec:add}, this is still an isomorphism for any $N$, in sufficiently high grading.  Thus, we obtain not just a single $U$-curve, but $N$ $U$-curves $C(i)$ in a row, i.e. 
\[ C(i) \in \mathcal{M}(\alpha(i),\alpha(i-1)).\]  
Then the total ECH index of these $N$ curves --- more precisely, the ECH index of the relative homology class $S_N$ given by their sum --- is $2N$, since the ECH index is additive under sums.  On the other hand, by the Weyl law, we can assume that $\mathcal{A}(\alpha(N)) \le O(\sqrt{N}).$  
We saw in \S\ref{sec:ujo} above that the $CZ^{top}$ term in \eqref{eqn:diff} is proportional to the action, at least for a single orbit.  If we are assuming finitely many simple orbits, then the same argument holds; a similar argument bounds the Chern class term under the assumption of finitely many simple orbits. With somewhat more care, this works equally well for nondegenerate perturbations of a contact form with finitely many orbits; and in fact, after the proof of Theorem~\ref{thm:hwz} it was shown in \cite{cgp} that these bounds hold as long as $c_1(\xi)$ is torsion, without any further assumptions.

From the above argument, we conclude that
\begin{equation}
\label{eqn:j0bound}
J_0(S_N) = I(S_N) + O(\sqrt{N}) = 2N + O(\sqrt{N}).
\end{equation}
Let us comment on the significance of this.   
Recall that we seek a bound on $J_0$, to get a bound on the topology of the non-trivial component of the curve via Proposition~\ref{prop:j0bound}.  The above equation can be interpreted as saying that on average, for $N$ sufficiently large, the curves have $J_0$ essentially $2$, by additivity of $J_0$.   This is a crucial point for our analysis.

Let us collect another key point that comes out of these kind of considerations.  A $J$-holomorphic curve $C$ has an {\em action} defined by
\[ \mathcal{A}(C) = \int_C d \lambda.\]
Recall that $d\lambda$ is pointwise nonnegative along $C$, so this is always nonnegative.  One can check that $d\lambda$ vanishes at a nonsingular point on $C$ if and only if the tangent space is spanned by the Reeb vector field and $\partial_s$.
By Stokes' theorem and the Weyl law, similarly as above, the sum of the actions of the $N$ curves $C(i)$ is $O(\sqrt{N})$.  Thus, on average the curves have action $O(1/\sqrt{N})$.  We therefore obtain that when $N$ is large, most of the curves have small action.  We call such curves {\em low-action} curves.  They play a key role in our argument, and also in the proof of Theorem~\ref{thm:hwz}.

Thus, to summarize, from the Weyl law and the arguments above, we learn the following:
\begin{itemize}
\item In a $U$-tower of length $N$, on average the curves have $J_0$ about $2$ for large enough $N$.  (In the argument, this requires the assumption that $c_1(\xi)$ is torsion, and it is an interesting question whether this is necessary)
\item In a $U$-tower of length $N$, on average the curves have ``low action", more precisely action $1/\sqrt{N}$.  
\end{itemize}

Let us note the following point, which will be one of the key challenges that we need to overcome: by the above considerations, we can find $J_0(C) \le 2$ $J$-holomorphic curves $C$ counted by the $U$-map.  However, in the argument we sketched in \S\ref{sec:outline} we desired cylinders.  It does not follow from Proposition~\ref{prop:j0bound} that a $J_0(C) = 2$ curve must have $C_1$ a cylinder, and in fact there are many examples where this fails.  For example, one could have $C_1$ with genus $0$ and four punctures, and $C_0$ empty.  

\subsection{Criteria for a global surface of section}

In order to continue with our argument, we should first clarify exactly what we are seeking.  Recall that Theorem~\ref{thm:hwz} would follow, if we can find a global surface of section (GSS).  We would like to produce this GSS by projecting a $J$-holomorphic curve.  The following is a special case of a more general criteria for a GSS for a nondegenerate contact form from \cite{CGHP}; the proof in \cite{CGHP} builds on ideas from 
\cite{fast,convex},
proving closely related results.

\begin{prop} \cite[Prop. 3.3]{CGHP}
\label{prop:jsection}
Let $C \in \mathcal{M}(\gamma_+,\gamma_-)$ be a $J$-holomorphic cylinder from $\gamma_+$ to $\gamma_-$ in $\mathbb{R} \times Y$.  Assume that:
\begin{enumerate}[(i)]
\item ind(C) = I(C) = 2
\item 
$\gamma_{\pm}$ are elliptic
\item The component $\mathcal{M}_C/\mathbb{R}$ of the moduli space of $J$-holomorphic curves containing $C$ is compact.
\end{enumerate}
Then $C$ projects to a global surface of section under the map $\mathbb{R} \times Y \to Y$.
\end{prop}  

Here, $ind(C)$ denotes the index of $C$, i.e. the expected dimension of the component of the moduli space containing $C$. 

\begin{proof}[Sketch]

Let us try to give an idea for how the proof goes in order to understand the role of the criteria (i) - (iii) (though this sketch can be skipped by a reader willing to take the proposition on faith.)  The general picture we are going for is that the curves in the same moduli space component of $C$ foliate $X = \mathbb{R} \times Y$, with embedded projection to $Y$.  As explained when we introduced ECH in \S\ref{sec:echdef}, the ECH index condition in (i) guarantees that $C$ is embedded in $X$.  However, to show that its projection remains embedded we need to know more.  To start to analyze this, we use an identity due to Wendl \cite{wendl}: there is a quantity called the {\em normal Chern number} of $C$, given for (not necessarily cylindrical curves) $C$ by the formula
\[ c_N(C) = \frac{1}{2}(2g(C) - 2 + ind(C) + h_+(u)),\]
where $h_+$ is the numbers of ends of $u$ at ``positive" hyperbolic orbits (namely, those such that the linearized return map has positive eigenvalues).  
One can show, see e.g. the exposition in \cite{CGHP}, that the normal Chern number bounds the number of zeros of the projection of $\partial_s$ to the normal bundle of an embedded curve $C$ from above.  One can show that since the curve is $J$-holomorphic, these zeros must all count positively; so, when $c_N(C) = 0$, the projection of $C$ to $Y$ is immersed.  In our case, $C$ is a cylinder so $g = 0$; item (i) implies that $ind = 2$ and then item (ii) implies that $h_+ = 0$.   

Thus, we learn from the assumptions that $c_N(C) = 0$ and then that the projection of $C$ is immersed.  To show that it is embedded it remains to show that this projection is injective.  Equivalently, we want to show that in the four-manifold $X$, $C$ can not intersect a translate of itself.  For this, we use some intersection theory, exploiting the very important fact that two $J$-holomorphic curves in a four-manifold must intersect positively.

In fact, we will show the more general fact that, under the assumptions in our proposition, two different curves $C$ and $C'$ in the same moduli space component can not intersect.  The basic idea behind this is as follows.  If $C$ and $C'$ were curves in a closed four-manifold, then their number of intersections would be counted by the homological intersection number, and in particular would be independent of the specific choice of $C'$ in the component of $C$.  Our setting is non-compact, but nevertheless there is a very well-studied ``relative intersection" theory, where one takes into account and studies the behavior of $C$ and $C'$ near infinity.  This is a subject with a long history --- we will not explain much here, except to note that these kind of considerations started in  \cite{props2}, with important developments in \cite{siefring, siefring2}.  Here, we note that the condition on the ECH index in (i) guarantees that, just as in the closed case, there is still a bound on the number of intersections between $C$ and $C'$, independent of the choice of $C'$; for the details, see for example Step 2 of \cite[Lem. 3.5]{CGHP}.  Once we know this, we can take $C'$ to be a small translate of $C$.   The idea is then to use the fact that the restriction of $\partial_s$ to $N$ has no zeros to conclude that $C$ does not intersect $C'$ for a small enough translation, except possibly at very large values of $s$; however, at such values, $C$ is converging exponentially fast to a Reeb orbit, and so by some further analysis $C$ still can not intersect $C'$.  

Thus, the projection of $C$ is embedded and also disjoint from the projection of any $C'$ in the same moduli space component.  Arguing similarly as above gives that in fact the projections of the $C'$ foliate some open subset of $Y$.  We would now like to show that we can take this open subset $U$ to be $Y \setminus \lbrace \gamma_+, \gamma_- \rbrace$. Let $\pi_Y$ denote the projection to $Y$.  We have $U = \cup_{C' \in \mathcal{M}_C} \pi_Y C'$ and this is open by our above analysis.  On the other hand, it is a closed subset of $Y \ \setminus \lbrace \gamma_+, \gamma_- \rbrace$ by the compactness assumption in (iii).  At the same time, we also know that $U$ is disjoint from $\lbrace \gamma_+, \gamma_- \rbrace$ by similar reasoning as above.

By compactness, the moduli space $\mathcal{M}_C/\mathbb{R}$ must be a circle, and so by above there is a map $f: Y \ \setminus \lbrace \gamma_+, \gamma_- \rbrace \to S^1 = \mathbb{R}/\mathbb{Z}$.  Because $J \partial_s = R$, our argument from above showing that $\pi_Y(C)$ is immersed applies to show that $\pi_Y(C)$ is transverse to the Reeb vector field, hence the derivative of $f$ in the Reeb direction is nowhere vanishing and we can assume without loss of generality that it is positive.  We will now show that there exists some $T$ such that for any $y \in \pi_Y C$, the total change in $f$ along the Reeb flow is at least $1$ in time $T$ and at least $-1$ in time $-T$; it follows immediately from this that $\pi_Y(C)$ is a GSS.  To prove this final claim, the only subtlety is to examine the behavior near $\gamma_{\pm}$.  In other words, we can take a neighborhood $N$ of $\gamma_{\pm}$ such that $Y \setminus N$ is compact, and then the derivative of $f$ in the Reeb direction is uniformly bounded from below on $Y \setminus N$ since it is nowhere vanishing.  On the other hand, near $\gamma_{\pm}$, there is a large body of asymptotic analysis, as summarized in \cite[Sec. 3]{CGHP}, that we can bring to bear on the problem.  From this, one finds, see \cite{CGHP}, that if the rotation number of $\gamma_+$ is $\theta$ in some trivialization, then the Reeb flow near $\gamma_+$ rotates relative to $C$ by about $\theta - \lfloor \theta \rfloor$ as one goes around $\gamma_+$; there is an analogous equation for $\gamma_-$.  Since the contact form is nondegenerate and both orbits are elliptic, $\theta$ is not an integer, so one finds that on $N$ there exists some $T$ that one can take as well.
\end{proof}

\subsection{Warm-up: low-energy cylinders for contact forms with finitely many orbits}
\label{sec:warm}

Before continuing, let us explain how to apply Proposition~\ref{prop:jsection} in a very special but illustrative case.  Many of the ideas here will motivate much of what follows.


To set the stage, let us assume that $\lambda$ is nondegenerate, with finitely many simple Reeb orbits.   Assume that $C_1$ is a non-trivial $J$-holomorphic cylinder from $\gamma_+$ to $\gamma_-$, detected by the $U$-map, in a length $N$ $U$-sequence.  Let us now attempt to apply Proposition~\ref{prop:jsection}.  We need to verify that items $(i) - (iii)$ hold.  The first item holds by the definition of the $U$-map.  As for the other two, the key point is as follows.  As explained in \S\ref{sec:utowers}, if $N$ is sufficiently large, then the action of a typical member of this $U$-sequence becomes as small as we like.  Now if the action of $C_1$  is sufficiently small, we claim that the following hold:
\begin{enumerate}[(i)]
\item $\gamma_+$ and $\gamma_-$ must be elliptic
\item $\mathcal{M}_C/\mathbb{R}$ is compact.
\end{enumerate}

The first item above holds because in the ECH chain complex, we only consider multiple covers of elliptic orbits; and, if there are only finitely many orbits, then a nontrivial cylinder with sufficiently small action must go between orbits of high multiplicity.  (If an orbit is an $m$-fold cover of a simple orbit, we call $m$ its multiplicity.)

The second item requires a little more work, and some more preliminaries.  We will state what we need as a lemma and review the preliminaries in the proof.

\begin{lemma}
\label{lem:warm}
Fix a nondegenerate contact form $\lambda$, with finitely many simple Reeb orbits.  Let $C_1$ be an $I = ind = 2$ $J$-holomorphic cylinder from $\gamma_+$ to $\gamma_-$.   If the action of $C_1$ is sufficiently small, $\mathcal{M}_C/\mathbb{R}$ is compact.
\end{lemma}

\begin{proof}

First we need to review how $\mathcal{M}_C/\mathbb{R}$ could fail to be compact.  There is a variant of Gromov's compactness theorem in this context, called ``SFT compactness" \cite{sftcomp}.  The version relevant here says that if the moduli space $\mathcal{M}_C/\mathbb{R}$, consisting of cylinders from $\gamma_+$ to $\gamma_-$ (in the component of $C$) is not compact, then there is convergence to a $J$-holomorphic {\em building} from $\gamma_+$ to $\gamma_-$ with at least $2$ non-trivial levels.  A $J$-holomorphic building from $\gamma_+$ to $\gamma_-$ in a symplectization $X$ is a finite sequence of $J$-holomorphic curves, each taking values in $X$, such that the asymptotics of consecutive curves match, and such that the positive asymptotics of the top level are $\gamma_+$ and the negative asymptotics of the bottom level are $\gamma_-$, see Figure~\ref{fig:broken} and see e.g. \cite{ECHlecture} for precise definitions; the number of curves in this sequence is the number of levels, and we assume that our levels are non-trivial (in the sense that they do not consist entirely of $J$-holomorphic cylinders that are $\mathbb{R}$-invariant.)
A reader familiar with Morse theory might find it useful to think of such buildings as somewhat analogous to the compactification of the space of flow lines by broken flow lines.  

In our situation, the curves in  $\mathcal{M}_C/\mathbb{R}$ have ECH index $2$.  A fundamental fact in the theory of ECH, see e.g. \cite{ECHlecture}, is that the ECH index in symplectizations is nonnegative, and it is $0$ only for (possibly branched) covers of trivial cylinders.  Thus, by additivity of the ECH index, to show that $\mathcal{M}_C/\mathbb{R}$ is compact, we need to show that, under our hypotheses, an $I=2$ cylinder can not break into two $I = 1$ curves.  

Let us consider some possible breakings of $C_1$ into a building containing two levels $C_+$ and $C_-$ with $I(C_{\pm}) = 1$.  If the action of $C_1$ is small enough, then no component of $C_{\pm}$ can be a disc.  This is because the action of such a disc would, by Stokes' theorem, have to be the action of some orbit; and, by making the action of $C_1$ sufficiently small, we can assume that it is smaller than the action of the shortest orbit.  Thus, since $C_{\pm}$ must glue\footnote{Strictly speaking, there could in addition be some levels with branched covers of trivial cylinders, but this does not affect the argument.} to give a cylinder, each $C_{\pm}$ must itself by an $I = 1$ $J$-holomorphic cylinder.  


Ruling out this breaking into two $I = 1$ cylinders $C_{\pm}$ requires several points.  Denote by $\gamma$ the negative asymptote of $C_+$, noting that by definition this is the same as the positive asymptote of $C_-$.  The first observation is that, by additivity of the action, the action of each of $C_+$ and $C_-$ must be even smaller than that of $C_1$.  Thus, as above, $\gamma$ must be an $m$-fold cover of some simple Reeb orbit, for a high value of $m$.  Moreover, $\gamma$ can not be elliptic: this is because the positive asymptote of $C_+$ (for example) is elliptic, and a simple index computation shows that if $\gamma$ were elliptic, $I(C_+)$ would have to be even.  

It remains to show that $\gamma$ also can not be hyperbolic.  For this we need an additional input from ECH theory.  In the basic theory of ECH, somewhere injective curves $C$ with $I(C) = ind(C)$ have special constraints on their asymptotics.  For an end (namely, the image of a small neighborhood of a puncture in the domain) of $C$ asymptotic to an orbit $\gamma$ of multiplicity $m$, call $m$ the {\em multiplicity of the end}.  It turns out that the condition $I = ind$ implies that the number and multiplicities of the positive (resp. negative) ends of $C$ at each simple Reeb orbit are determined by the total multiplicity of all positive (resp. negative) ends at the orbit, and the rotation number of that orbit.  As an example, let us state how this works in the hyperbolic case, which is all we will need.  


\begin{lemma}[Partition conditions in the hyperbolic case] \cite{ir, ECHlecture, HutJEMS}
\label{lem:parth}
Let $C$ be a somewhere injective $J$-holomorphic curve in $X$ with $I(C) = ind(C)$.  Assume that $C$ has positive ends at a simple hyperbolic orbit $\gamma$, of total multiplicity $m$.  Then:
\begin{enumerate}[(i)]
\item If $\gamma$ is positive hyperbolic, $C$ has $m$ distinct ends, each of multiplicity $1$.
\item If $\gamma$ is negative hyperbolic and $m$ is even, $C$ has $m/2$ distinct ends, each of multiplicity $2$.
\item If $\gamma$ is negative hyperbolic and $m$ is odd, $C$ has $\lfloor m/2 \rfloor$ distinct ends, each of multiplicity $2$, and a single end of multiplicity $1$.
\end{enumerate}
\end{lemma}

Similar (but more involved) formulas hold in the elliptic case, and we will return to this later.

With Lemma~\ref{lem:parth} in hand, it is now not difficult to complete the argument.   We wanted to show that $\gamma$ can not be hyperbolic.  We know that its multiplicity can be assumed as high as we like by taking the action of $C_1$ sufficiently small.  On the other hand, we know that $C_{\pm}$ are cylinders with $I = ind = 1$.  However, by Lemma~\ref{lem:parth}, any curve with $I = ind$ with ends at $\gamma$ must have at least $2$ ends when $\gamma$ is at least a $3$-fold cover.

\end{proof}

To sum up and recap the key point: in this special but very representative case, where the contact form is nondegenerate and there are finitely many simple orbits, we are done if we can find cylinders in sufficiently long $U$-sequences.  The challenge then becomes to find such cylinders.  This will be perhaps the main challenges in our proof and we will need to do this in a more general settings, namely without the nondegeneracy hypothesis, or for nondegenerate perturbations close to a degenerate form with only finitely many simple orbits.  How to do this will be the topics of the next section.

 \begin{figure}
  \centering
  \includegraphics[width=6cm]{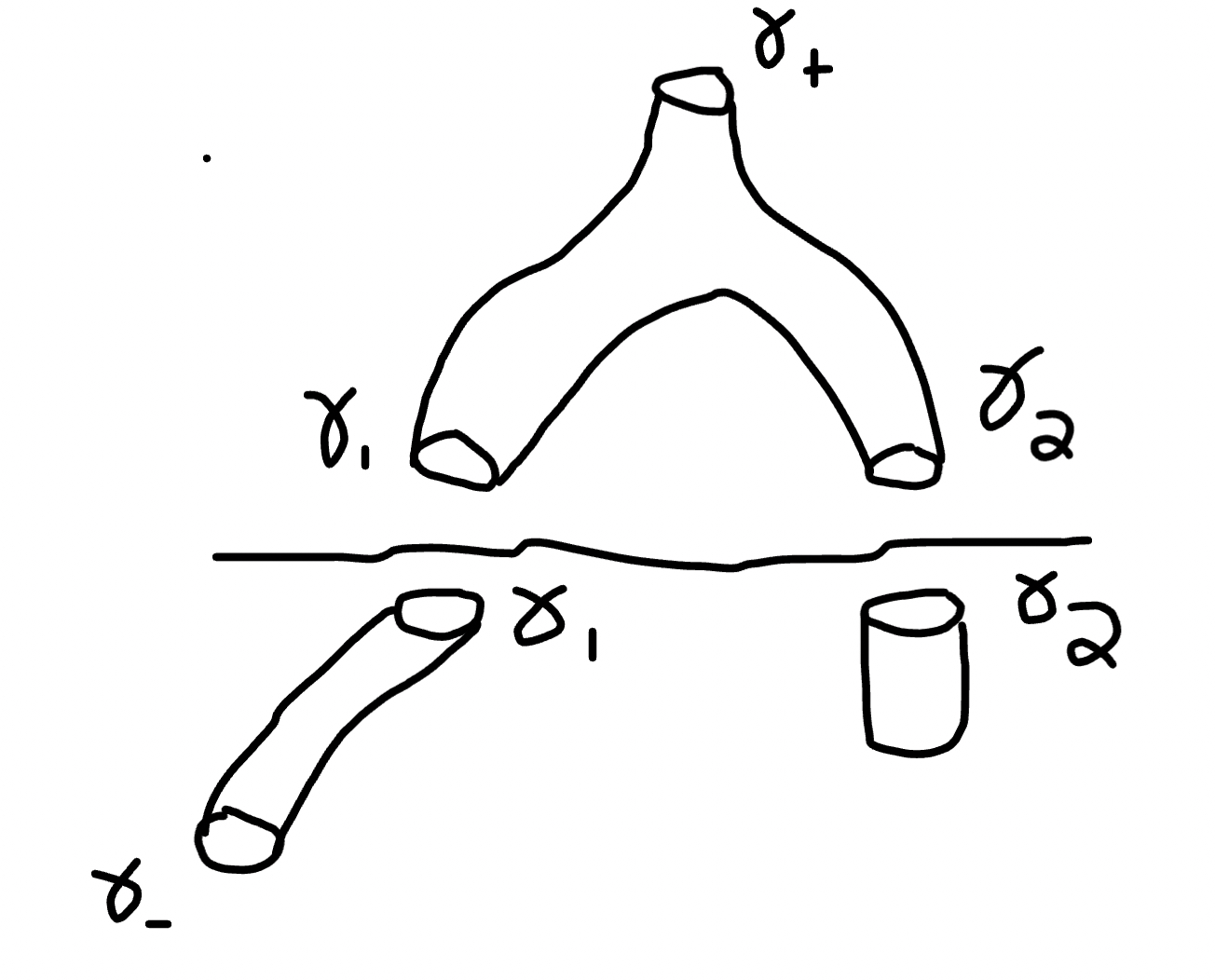}
  \def\svgwidth{,6 \textwidth}
  \caption{A two-level broken $J$-holomorphic curve.  This could be the degeneration of a cylinder from $\gamma_+$ to $\gamma_-$.}
  \label{fig:broken}
\end{figure}

\subsection{Perturbations and limits}

To go from the ``warm-up" in the previous section to Theorem~\ref{thm:hwz}, we would like to resolve the following challenges:
\begin{enumerate}[(i)]
\item Prove a reasonable analogue of Proposition~\ref{prop:jsection}
for a contact form $\lambda$ with finitely many simple periodic orbits, without any nondegeneracy assumption.
\item Find low action $J$-holomorphic cylinders for $\lambda$
\end{enumerate}

In this section, we will explain the general scheme for addressing item (ii) above, clarify the challenges, and start to get to work explaining how to resolve them.

Recall that ECH is not defined for $\lambda$ directly, because it is degenerate.  To find such cylinders,  the idea is to perturb $\lambda$ to nondegenerate contact forms $\lambda_n \to \lambda$ and then study $U$-towers for $\lambda_n$.  The following lemma tells us what we will want to know about the $\lambda_n$. 

\begin{lemma} \cite[Lem. 3.5]{CGHHL2}
\label{lem:pert}
Let $\lambda$ be a contact form with finitely many simple periodic orbits $\gamma_1, \ldots, \gamma_k$.    If $n$ is sufficiently large, then  the simple periodic orbits of $\lambda_n$ with action $\le n$ are as follows:
\begin{itemize}
\item Near a periodic orbit $\gamma_i$ with rotation number $p/q$, the simple periodic orbits
are either $C^{\infty}$-close to $q$-fold covers of $\gamma_i$, or $\gamma_i$ itself.
\item Near a nondegenerate periodic orbit of $\gamma_i$, the only simple periodic orbit
is $\gamma_i$ itself. 
\end{itemize}
\end{lemma}

Lemma~\ref{lem:pert} is proved by a careful local analysis around the orbits, using a lemma due to Bangert \cite{Bangert}.

Thus, the picture one should have is that a single simple degenerate orbit can give rise to many orbits, but these are all close to $q$-fold covers, up to long orbits.  By working below some fixed action threshold, we can ignore these long orbits, and so we get a quite specific structure for nondegenerate forms near $\lambda$.  The first important point is that this structure is similar enough to the ``warm-up" case from \S\ref{sec:warm} that the arguments there work for this situation with only minor adaptations; we leave this as an exercise to the reader.  


Assume now that we can find $J_n$-holomorphic cylinders $C_n$ for $\lambda_n$, such that each cylinder is asymptotic to Reeb orbits with action with an $n$-independent upper bound, and such that these cylinders project to global surfaces of section for $\lambda_n$.   (This is challenging, and we will say more about how to find such cylinders in subsequent sections.)  We would like to extract a useful cylinder for $\lambda$ in the limit.  To clarify what goes into this, it is useful to say more about the asymptotics of $J$-holomorphic curves.  We focus on the positive ends; there is a similar story for the negative ends.  As explained in e.g. \cite{CGHHL2}, by \cite{props1}, a positive end of a $J$-holomorphic curve converges exponentially fast to a Reeb orbit.  More precisely, if the end is at the $m$-fold cover of a simple Reeb orbit $\gamma$, and we use the trivialization $\tau$ to identify a neighborhood of $\gamma$ with $S^1 \times D^2$, the end is the image of the map
\[ [s,\infty) \times \mathbb{R}/m\mathbb{Z} \to \mathbb{R} \times (\mathbb{R}/\mathbb{Z}) \times D^2,\]
\[ (s,t) \to (s, \pi(t), \eta(s,t)).\]
Here, $\pi$ is the natural projection, and $\eta$ satisfies
\begin{equation}
\label{eqn:convergence}
\eta(s,t) \sim e^{\mu s} \phi(t),
\end{equation}
for a certain function $\phi(t)$ and corresponding number $\mu$.  More precisely, there is an operator associated to $\gamma^m$, called the ``asymptotic operator", $\mu < 0$ is an eigenvalue of this operator, and $\phi$ is the corresponding eigenfunction, which is nowhere vanishing.

In our case, what will be crucial is the following.  Each $C_n$ has an associated eigenvalue $\mu_n$ in \eqref{eqn:convergence}.  We will want the $\mu_n$ to stay away from $0$.  Then, arguments as in \cite{elliptic, HSW} will guarantee a well-defined exponential convergence in the limit.  Let us state this as the following lemma.  The proof, which we will also explain, makes crucial use of the previously introduced partition conditions.

\begin{lemma}
\label{lem:key}
Let $C_n$ be a sequence of $J_n$-holomorphic cylinders with $I = ind = 2$ for nondegenerate contact forms $\lambda_n$, asymptotic to elliptic orbits with an $n$-independent bound on the action.  Assume that the multiplicities of these orbits are at least $2$.  Define $\mu_n$ by \eqref{eqn:convergence}.  Then there exists $c > 0$ such that  $|\mu_n|> c$. 
\end{lemma}  

\begin{proof}

The idea is to look at the winding number of the corresponding eigenfunctions around $0$.  By the work in e.g. \cite{convex, props1, props3}, this is connected to the Conley-Zehnder index as follows.  In the nondegenerate case, when the Conley-Zehnder index of the trivialized orbit $\gamma$ is odd, we have that $CZ(\gamma) = 2w+1$, where $w$ is the winding number associated to the largest negative eigenvalue.  (There is a similar interpretation in the nondegenerate case when $CZ$ is even, but we will not need this.)   In the degenerate case, one can define $CZ(\gamma)$ to be the minimum Conley-Zehnder index of a nondegenerate perturbation. 


Now let $\gamma_n$ denote the Reeb orbit at the positive asymptotics of $C_n$.  By the bound on the action, the $\gamma_n$ are converging to a Reeb orbit $\gamma$ for $\lambda$ (after passing to a subsequence if necessary), and the case that does not follow immediately from standard theory is the case when $\gamma$ is degenerate; this is the case we will explain.  In this case, it is convenient to assume that the rotation number for $\gamma$ is $0$.  We can guarantee this by changing our trivialization as necessary, and this induces a trivialization for $\gamma_n$.

We now claim that $CZ(\gamma_n) = -1$, which is the crucial fact for which the partition conditions are key.  The point is as follows.  Recall that   
\[ CZ(\gamma_n) = \lfloor \theta_n \rfloor + \lceil \theta_n \rceil,\]
where $\theta_n$ is the rotation number.  In our case, this rotation number is close to $0$.  So, there are two possibilities: $\theta_n$ is either slightly positive or slightly negative.  The claim follows once we show that $\theta_n$ must be slightly negative.  This is where the partition conditions come in.  Similarly to \S\ref{sec:warm}, we will need just a special case:


\begin{lemma}[Partition conditions for nondegenerate elliptic perturbations] \cite{ir, ECHlecture, HutJEMS}
\label{lem:parte}
Let $C$ be a curve with $I(C) = ind(C)$.  Assume that $C$ has positive ends of total multiplicity $m$ at a simple elliptic orbit $\gamma'$ with rotation number (modulo one) $\theta' \in [0,1)$ (mod 1) satisfying 
\[ \theta' < 1/m.\] 
Then $C$ has $m$ distinct ends, each of multiplicity $1$.
\end{lemma}

In our case, our curves are cylinders, so we learn from Lemma~\ref{lem:parte} that $\theta'$ can not be in $[0,1/m)$, thus must be slightly negative, as desired.

Thus $CZ(\gamma_n) = - 1$ and we learn from this that the winding number $w_n$ associated to the largest negative eigenvalue is $-1$.  We also learn that $CZ(\gamma) = -1$, and so the winding number $w$ of the largest negative eigenvalue for $\gamma$ must be $-1$.   If the $\mu_n$ were converging to $0$, then an eigenfunction with eigenvalue $0$ would have negative winding number.  However, it is also known, as in \cite{convex, props1, props3}, that the winding numbers are monotonic in the eigenvalues;  thus, since $w = -1$, the winding numbers for the $0$ eigenvalue must all be nonnegative, which is a contradiction.




\end{proof} 

If the cylinders $C_n$ we find are found in a $U$-sequence, then along the lines of the arguments we gave in \S\ref{sec:warm}, one can guarantee that they satisfy the hypotheses of Lemma~\ref{lem:key}.  So, to summarize, to resolve item (ii) at the beginning of this section, the main point is to find such cylinders for the perturbations $\lambda_n$ from Lemma~\ref{lem:pert}.  We will explain this in \S\ref{sec:cylin}.

\subsection{Exponential weights}

Let us now discuss item (i) from the previous section: finding an analogue of Proposition~\ref{prop:jsection} for a contact form $\lambda$ with finitely many orbits.  The idea for this is to consider a moduli space $\mathcal{M}(\gamma_+,\gamma_-)$ of $J$-holomorphic cylinders from $\gamma_+$ to $\gamma_-$, with exponential convergence as in \eqref{eqn:convergence}, but without requiring $\gamma_{\pm}$ to be nondegenerate.  As explained in the previous section, to show that this moduli space is nonempty one would just need to find cylinders for perturbations $\lambda_n$ along the lines of Lemma~\ref{lem:key}.  If this is successful, we would find not just a $C \in \mathcal{M}(\gamma_+,\gamma_-)$ but one that arises as the limit of a sequence of global surfaces of section for $\lambda_n$.  

One can then attempt to mimic the arguments in Proposition~\ref{prop:jsection} and we will explain the idea behind this.  The first point is that even though $\gamma_{\pm}$ are not necessarily nondegenerate, one can use the exponential convergence built into the moduli space to nevertheless show that $\mathcal{M}(\gamma_+,\gamma_-)$ is a manifold.  This goes back to work of Hofer-Wysocki-Zehnder in \cite{props3}, who define a Fredholm theory with exponential weights that applies to such a setting.  One can similarly define a ``weighted" Conley-Zehnder index, see e.g. the exposition in \cite{CGHHL2}.  The other hypotheses in Proposition~\ref{prop:jsection} can be suitably modified, using the exponential convergence.  One can even prove an analogue of the SFT compactness statement introduced in \S\ref{sec:warm} that suffices for our purposes.  One could then try to check that a $C \in \mathcal{M}(\gamma_+,\gamma_-)$ that arises from a limit of global surfaces of section satisfies these modified assumptions.   

For example, let us explain why compactness would hold, which is perhaps the hardest point. If the moduli space $\mathcal{M}(\gamma_+,\gamma_-) / \mathbb{R}$ is not compact, then by the promised variant of SFT compactness, one finds a Reeb orbit in the breaking linking positively with the relative homology class of $C$.  However, for $C_n$ close to breaking along $C$, this would give a loop in the image of the projection of $C_n$, with positive linking number, contradicting the fact that the $C_n$ are global surfaces of section.

Roughly speaking, all of the above ideas can be made to work, and this is the content of Chapters 4 and 5 of \cite{CGHHL2}.  Some of this requires extensive technical work, so in these notes, we will focus on some other matters, and so we will not say much more about the details, referring the reader to \cite{CGHHL2}.  One point that we do want to make, though, is as follows.  In practice, some work is required to show that the promised $C \in \mathcal{M}(\gamma_+,\gamma_-)$ is somewhere injective.  This might be true, but instead of addressing this, at the very end of our argument we implement a short workaround.  Namely, we do show that there is a compact moduli space $\mathcal{M}(\gamma_+,\gamma_-) / \mathbb{R}$ mapping to $Y$, but rather than actually produce a GSS from this, we lift to the moduli space $S^1 \times S^1 \times \mathbb{R}$ via an evaluation map, and show that the lifted flow has a global surface of section instead; see Chapter 7 of \cite{CGHHL2}.  

\subsection{Finding cylinders}
\label{sec:cylin}

In view of the discussion in the previous sections, what remains to be explained is why the contact forms $\lambda_n$ have $J_n$-holomorphic cylinders projecting to global surfaces of section.  This requires quite intricate analysis of the ECH $U$-map towers, and explaining the key ideas behind this will be the topic of the remainder of the section.  We divide our explanation into several parts.  We are assuming (unless we note otherwise) that throughout we have fixed a nondegenerate contact form $\lambda_n$ as in Lemma~\ref{lem:pert}.

\subsubsection{Preliminary considerations and the basic idea for finding cylinders}

Recall from \S\ref{sec:utowers} that if we take a length $N$ $U$-sequence for $\lambda_n$, then on average we expect a curve $C$ to have $J_0(C) = 2$ and to have small action.  On the other hand, recall from Proposition~\ref{prop:j0bound} that a curve $C$ with $J_0 = 2$ need not have $C_1$ a cylinder.  (Recall the decomposition $C = C_0 \cup C_1$.)  Indeed, it is helpful to start by cataloguing the various options for a $J_0 = 2$ curve counted by the $U$-map.  If $C$ has sufficiently low action, then it must have both a positive and a negative end, and this allows us to rule out various possibilities; for example, a  curve with $C_1$ having genus $0$ and two positive ends at the same orbit, with $C_0$ also having ends at that orbit would have $J_0 = 2$, but not low action.  To warm-up, let us list some possibilities that could happen after ruling out low action curves:

\begin{enumerate}[(i)]
\item We could have $C_1$ having genus $1$ with two ends.
\item We could have $C_1$ having genus $0$ with three ends at distinct orbits.  In this case, $C_0$ would have to also have ends at one of these orbits.
\item We could have $C_1$ a cylinder from one orbit to another, with $C_0$ having ends at both orbits.  This is the case that we want to occur.
\item We could have $C_1$ having genus $0$ with four ends, all at distinct orbits.
\end{enumerate} 

Here is the key takeaway from this: the relationship between $C_1$ and $C_0$ is crucial.  Namely, we have the following observation that we leave as an exercise for the reader.

\begin{exe}
Let $C$ have $J_0 = 2$ and assume that for every orbit at which $C_1$ has ends, $C_0$ also has ends.  Show that if the action of $C$ is sufficiently small, then $C$ is a cylinder with one positive end and one negative end.
\end{exe}

In view of the above exercise, we therefore have to show that the curves in a $U$-tower have ``enough" trivial cylinders.  Remember that we require just one low action cylinder, in a length $N$ $U$-sequence, so what we can assume is that there are never enough trivial cylinders and try to derive a contradiction. To do this, we need to find something we can leverage for curves $C$ without trivial cylinders, and we return to the aforementioned partition conditions for this.  Let us now explain how this works.

\subsubsection{Partition conditions, revisited} 

To continue, we need to explain the partition conditions in more depth.  Let $C$ be a somewhere injective $J$-holomorphic curve with $I = ind$.  Assume that $C$ has positive ends at some orbit $\gamma$ of total multiplicity $m$, and assume that the rotation number of $\gamma$, after fixing some trivialization $\tau$, is $\theta$.  Then, the more general version of the partition conditions that we will need state that the number of ends, and their multiplicities, depends only on $\theta$ and $m$; we denote this number by $p^+_{\theta}(m)$ and it is determined combinatorially by $\theta$ and $m$.  There is an analogous story for the negative ends.  

Explicitly, to compute $p^+_{\theta}(m)$, we look at the line $y = m \theta$, and take the maximal piecewise linear path with vertices on integer lattice points, staying below this line, and connecting $(0,0)$ to $(m,\lfloor m \theta \rfloor)$.  Then, $p^+_{\theta}(m)$ is given by the horizontal displacements in this path.   

Here are some facts about the partition conditions that we will need:

\begin{lemma}\cite{CGHHL2, ECHlecture}
\label{lem:partrev}
\begin{enumerate}[(i)]
\item $p^+_\theta(m)$ and $p^-_\theta(m)$ are disjoint when $m > 1$.
\item $1 \in p^+_\theta(m)$ if and only if $1 \not \in p^-_\theta(m)$ when $m > 1$.  
\item $|p^+_\theta(m)| + |p^-_\theta(m)| \le 3$ only if $m \lbrace \theta \rbrace  < 2$ or $m (1 - \lbrace \theta \rbrace) < 2$.
\end{enumerate}
\end{lemma}

For our purposes, the way we can think about the above lemma is as follows.  Remember that we are interested in highlighting facts in the absence of enough trivial cylinders in a $U$-sequence, in order to get a contradiction.  The first item above says that if $C_1$ has positive ends at some orbit $\gamma$ with total multiplicity $m$ and $C_0$ has no ends there, then no other curve  in the $U$-sequence can have  $C_1$ with negative ends at $\gamma$ with multiplicity $m$ and $C_0$ having no ends there.  When there are finitely many orbits, this can be quite powerful, since a long $U$-sequence will have to have ends on a particular orbit many times.  The second item (as well as the first) will be relevant when we discuss ``the score" below.  As for the third item, one can think of it as saying that, when the partition conditions hold for $C_1$, there tend to be at least four ends except in some very highly constrained situations which one can often rule out for example by forcing $m$ to be quite large; this in turn is useful when one knows that $J_0$ is on average about $2$.

\subsubsection{The score}

To start to make the above ideas precise, it is very useful to define a quantity called ``the score".  We now explain how this works.

The first point involves Lemma~\ref{lem:partrev}.(iii).  If we knew that $\lambda_n$ had finitely many simple Reeb orbits, then for sufficiently large $m$, Lemma~\ref{lem:partrev}.(iii) would guarantee that  $|p^+_\theta(m)| + |p^-_\theta(m)| \ge 4$.  Our perturbations $\lambda_n$ do not quite have the property of finitely many Reeb orbits, but by Lemma~\ref{lem:pert} and Lemma~\ref{lem:parte} we can still assume that $|p^+_\theta(m)| + |p^-_\theta(m)| \ge 4$ for sufficiently large $m$, after disregarding ``long orbits"; this is because up to long orbits, all of the new orbits still are close to $q$-fold covers of orbits with rotation number $p/q$.  

With that in mind, let us choose $M$ such that if $m > M$, we have $|p^+_\theta(m)| + |p^-_\theta(m)| \ge 4$.  Now let $\alpha = \lbrace (\gamma_i, m_i) \rbrace$ be an orbit set.  We call a particular element $(\gamma_i,m_i)$ of this set a {\em component} of the set.  We say that a component is a {\em $p^+$-component} if $p^+_{\theta_i}(m_i) = m_i$; here $\theta_i$ is the rotation number of $\gamma_i$.  We define a $p^-$ component analogously.  Finally, we define a {\em special component} to be a $(\gamma_i,m_i)$ such that $m_i > 1$ but $1 \not \in p^+_{\theta_i}(m_i)$.  These definitions are inspired by the facts about the partition conditions in Lemma~\ref{lem:partrev}.  We let $p^+(\alpha)$ denote the number of $p^+$ components, and define $p^-(\alpha)$ and $s(\alpha)$ analogously.

We now define the {\em orbit score} of $\alpha$ by 
\begin{equation}
\label{eqn:scoredef}
S(\alpha) := p^+(\alpha) + s(\alpha) - p^-(\alpha)
\end{equation}
and we define the {\em orbit score} of a $J$-holomorphic curve $C$ counted by the $U$-map from $\alpha$ to $\beta$ by
\begin{equation}
\label{eqn:scoredefcurve}
S(C) := S(\alpha) - S(\beta). 
\end{equation}

The motivation for the definition is as follows.  Recall that we are considering $N$ curves in a row counted by $U^N$ and we want to show that there are ``enough" trivial cylinders.  The idea is that when there are no trivial cylinders for some curve in this sequence, the score should go up.  For example if for such a level, say from some $\alpha'$ to some $\beta'$, $C_1$ has a positive end at some orbit but there are no trivial cylinders, then it follows from the partition conditions and Lemma~\ref{lem:partrev} that the corresponding component of $\alpha'$ is a $p^+$ component and can not be a $p^-$-component.  The other quantities play a similar role.

Finally, though on average $J_0 = 2$ by the considerations explained above, we would like to measure any possible deviation from this average.  So, we set $y(C) = J_0(C) - 2$, and we define the {\em total score} of C by
 \begin{equation}
\label{eqn:scoredefcurvet}
T(C) := S(C) + 3y.   
\end{equation}
The basic idea is that since we just want to find cylinders, curves with $J_0 \ge 3$ are less useful for our purposes than curves with $J_0 \le 1$ (which is  a stronger condition than $J_0 = 2$ for the purposes of finding cylinders), so we want curves with $J_0 > 2$ to also have a high score.

To sum up, then, the score is defined so that the ``bad" curves that do not give us what we want should have a high score.  Then, the idea is that as long as we can bound the sum of the scores of all the curves in the length $N$ $U$-sequence, we get a bound on the number of bad curves.

\subsubsection{Nonnegativity of the score}

Recall that we would like to show that ``bad" curves, i.e. curves that are not suitable for finding our desired GSS, make a positive contribution to the total score $T$.  The following lemma gives just enough for our purposes.  Recall the constant $M$ from the previous section.

\begin{lemma} \cite[Lem. 3.22]{CGHHL2}
\label{lem:totalscore}
Let $C$ be a $J_n$-holomorphic curve counted by the $U$-map and write $C = C_0 \cup C_1$ as usual.  Assume that $C$ has sufficiently small action and is not a cylinder.
Then:
\begin{enumerate}[(i)]
\item $T(C) \ge 0$.
\item If $T(C) = 0$ and $J_0(C) \ge 2 $ then all of the following occur:
\begin{itemize}
\item $J_0(C) = 2$
\item $C_1$ has genus $0$ and exactly three ends.
\item There is an orbit $\gamma_i$ at which $C_1$ has a positive end of multiplicity $\ge M$  and there is a different orbit $\gamma_j$ at which $C_1$ has a negative end of multiplicity $\ge M$.  
\item $C_0$ has ends at $\gamma_i$, of multiplicity at least $2$, and this is the only orbit at which both $C_0$ and $C_1$ have ends.
\item Either $C_1$ has an additional positive end, at an orbit different from $\gamma_i$, of multiplicity one; or, $C_1$ has an additional negative end, at an orbit different from both $\gamma_i$ and $\gamma_j$.  
\end{itemize}
\item If $T(C) = 0$ and $J_0(C)  \le 1$ then $J_0(C) = 1$. 
\end{enumerate}
\end{lemma}

While the lemma is a bit technical, the following picture should be somewhat more digestible.  We wanted to, ideally, show that all ``bad" curves -- i.e. the curves that will not suffice for our purposes -- make a positive contribution to the score.  Instead what we find is that all curves at least make a nonnegative contribution to the score; and, the ones that make zero contribution are quite special.  The ones that make contribution zero that are not cylinders will not suffice for our needs regarding finding a GSS, but they are so specific that we can rule them out with some related arguments and we will say a little bit more about this in the next section.

The proof of Lemma~\ref{lem:totalscore} involves a careful case--by--case analysis of a combinatorial flavor, going orbit by orbit, of how $C_0$ and $C_1$ could interact, using the formula Proposition~\ref{prop:j0bound}. The other key input is a further fact about the partition conditions, that in the case where $C_0$ and $C_1$ both have positive ends at some orbit $\gamma$, say of multiplicity $m$ and $n$ respectively then 
\[ p^+_{\theta}(m+n)  = p^+_{\theta}(m) \cup \lbrace a_1, \ldots, a_k \rbrace,\]
where $\lbrace a_1, \ldots, a_k \rbrace$ is the set of multiplicities of the ends of $C_1$ at $\gamma$; see e.g. the summary in \cite{CGHHL2}.  

The kind of considerations mentioned when we motivated the score play a crucial role.  For example, when there are no trivial cylinders where $C_1$ has ends, one gets either a large positive contribution to $J_0$ if there are many ends, hence the $y$ term in \eqref{eqn:scoredefcurvet}, or positive contributions to $p^+$ and $p^-$.  There are many kinds of patterns that occur, but the idea is that other than a small finite set of possibilities which can be analyzed explicitly, there are good enough estimates to conclude that $T(C) > 0$.  As there are many cases, we refer the reader to \cite[Sec. 3.3]{CGHHL2} to see how this is all worked out.

\subsubsection{Putting it all together}

Let us now sketch how to use the above to find the desired cylinder for $\lambda_n$, projecting to a global surface of section. 

\begin{proof}[Sketch of proof]

Our focus is on conveying the main ideas and so we will be a bit brief with the proof.  
The detailed proof appears in \cite[Sec. 3.4]{CGHHL2}.  

Take a length $N$ $U$-sequence $C(1), \ldots, C(N)$, where $C(i)$ is a curve from $\alpha(i)$ to $\alpha(i-1)$ and assume that none of these curves are the low-action cylinders that we want.   Remember that we should think of $N$ as large.  The first point is that Lemma~\ref{lem:totalscore} only applies to curves of sufficiently low action.  Low action is also necessary for our arguments guaranteeing that the projection is a global surface of section.   As explained  in Section~\ref{sec:utowers}, in fact most curves have low action, more precisely, the number failing to have low action has an $O(\sqrt{N})$ bound.  So, we begin by noting this.

The idea is then to consider the other kinds of curves in our length $N$ sequence that would be problematic, and build similar bounds.  


For example, let us next consider the set of curves with low energy but $T > 0$.   Because it is a telescoping sum, we have 
\[ \sum^N_{i=1} T(C(i)) = S(\alpha(N)) - S(\alpha(0)) + 3 \sum^N_{i=1} y(C(i)).\]
Now the terms with $S$ are proportional to the length of the orbit set.  Hence, both are $O(\sqrt{N})$ by the Weyl Law, Theorem~\ref{thm:weyl}.  Similarly, the argument in Section~\ref{sec:utowers} for showing that $J_0$ is on average about $2$ gives that $\sum^N_{i=1} y(C(i))$ is about $O(\sqrt{N})$ as well.  Thus, the sum on the left hand side of the above equation has an $O(\sqrt{N})$ bound.  Since we showed in Lemma~\ref{lem:totalscore} that the low action curves all have nonnegative score, and since the high action curves also have an $O(\sqrt{N})$ bound (and still have the very coarse bound that each can contribute no more than, say, $-10$ to the score), an $O(\sqrt{N})$ bound on the sum gives that not many of the $C(i)$ can have positive score.  To make this precise, one can extract an $O(N^{2/3})$ bound.  (The appearance of the exponent $2/3$ is definitely not optimal, but suffices for our purposes.)

One can similarly bound the number of curves $C(i)$ with $T = 0$ and $J_0 = 1$ and then the number of curves $C(i)$ with $T = 0$ and $J_0 = 2$, getting (for example) $O(N^{3/4})$ and $O(N^{4/5})$ bounds respectively. At each stage, one makes use of the previous bounds, just like as in the argument for $T > 0$, when we needed to use the previous bond on the number of curves that do not have low action.  Getting this to work requires replacing the total score in the telescoping sum argument with slight variants.  Specifically, for the $(T,J_0) = (0,1)$ bound, it is convenient to use the quantity $T'(C) = S(C) + 2y$, which differs from $T(C)$ by $y(C)$; this takes advantage of the fact that in this case $y$ is negative.  For the $(T,J_0) = (0,2)$ bound, we need an invariant that is well-suited to cleaning up the very specific remaining cases in Lemma~\ref{lem:totalscore}.  For this, we define the {\em $K$-invariant} for a curve $C$ from $\alpha$ to $\beta$,
\[ K(C) = K(\alpha) - K(\beta) + 2 y(C),\]
where $K(\alpha) \le 0$ is minus the number of components $(\gamma,m)$ with $m > 1.$ Then, for the specific curves $C'$ with $(T,J_0) = 2$ classified in Lemma~\ref{lem:totalscore}, one can check case by case that $K(C') \ge 1$.  

Thus, under our assumption that there are none of the curves that we wanted, we start with $N$ curves, and show that all of them have some bound of the form $O(N^{4/5})$.  This is a contradiction.
\end{proof}

Thus we have now given a sense for the key ideas behind the various ingredients for the proof of Theorem~\ref{thm:hwz}.  Before moving to the next chapter, we should make one remark about constants.  To recall, to prove Theorem~\ref{thm:hwz}, we approximate $\lambda$ by contact forms $\lambda_n$, and take global surfaces of section for $\lambda_n$ by projections of $J_n$-holomorphic curves $C_n$.  In our arguments, there are implied bounds, for example ``sufficiently low action curves", and one needs to make sure that the constants implemented to make this rigorous can be described in an $n$-independent way.  This is doable with some care.  Perhaps the most important point in this regard is the continuity of the ECH spectral invariants explained in \S\ref{sec:add}.  The idea is essentially to phrase all needed bounds in terms of the ECH spectral invariants of $\lambda$, and the length of the shortest orbit for $\lambda$.  These are both $C^{\infty}$ continuous, and so are well-suited to building bounds for the quantities that come up in the arguments for finding cylindrical global surfaces of section for $\lambda_n$.  This also guarantees that all the cylinders $C_n$ have $n$-independent upper and lower bounds on their action, which is important when we apply the SFT compactness type machinery in \S\ref{sec:warm} as explained in that section.  


\section{Remarks on the proof of the Le Calvez - Yoccoz property}
 \label{sec:ley}
  
 Let us say a few words about the proof of the aforementioned Theorem~\ref{thm:ley}.  Recall that this says that, under the hypotheses of the theorem, the complement of a minimal set is never minimal.  Because our lectures have limited time, we will be brief about this, trying to highlight a few key points.
 
 \subsection{A new compactness theorem}

The first point is a new way to detect closed non-trivial invariant sets beyond periodic orbits via $J$-holomorphic curves.   Our theorem will work in any dimension.  We will focus on the case of Reeb flows, though this aspect of the story works much more generally.  
Recall the action $\mathcal{A}(C_k) = \int_{C_k} d \lambda$ of a $J$-holomorphic curve.  Now consider a sequence of $J$-holomorphic curves $C_k$ in $\mathbb{R} \times Y$ (we can in fact let $J$ vary as well, but for simplicity we will just fix a $J$), with actions tending to zero.  We define the {\em limit set} $\mathcal{L}$ of the sequence to be the set of all closed subsets arising as (subsequential) Hausdorff limits of height two slices of the $C_k$.  By translating these height two sets, we view the limit set as a subset of the set of closed subsets of $(-1,1) \times Y.$       


\begin{theorem}\cite[Thm. 6]{cgp}
\label{thm:limit}
Let $C_k$ be $J$-holomorphic curves in $\mathbb{R} \times Y$ with actions tending to $0$ and limit set $\mathcal{L}$.  Assume in addition that $\chi(C_k)$ has a $k$-independent lower bound.
Then any $S \in \mathcal{L}$ satisfies
\[ S = (-1,1) \times K\]
for some nonempty invariant set $K \subset Y$.  
\end{theorem}    
 
The proof of Theorem~\ref{thm:limit} builds on ideas from the aforementioned feral curve theory of Fish-Hofer \cite{fh}.
It does not require that $Y$ is three-dimensional.  It also works for much more general structures than stable Hamiltonian ones; for example, it works for any smooth energy level of a Hamiltonian in a symplectic manifold; see \cite{cgp} for the precise statement.

Let us now try to give a sense for how Theorem~\ref{thm:limit} is proven.  The key point is the following ``local-area" bound.  Consider a $J$-holomorphic curve $C$ in $X = \mathbb{R} \times Y$, of low action.  The contact form induces an $\mathbb{R}$-invariant Riemannian metric and we will be interested in the area of pieces of $C$ in small balls with respect to this metric.  More precisely, for any $r$, let $A_x(r)$ be the area of the connected component of $C$ containing $x$ in the ball of radius $r$ about $x$ in $X$.  Here is the promised bound:

\begin{proposition}
\label{prop:bound}
There exist positive constants $\epsilon_1$ and $\epsilon_2$ with the following significance.
Let $C$ be a $J$-holomorphic curve in $X$.  Then
\[ A_x(r) \le \epsilon_2(\chi(C)^2 + 1)\]
whenever $\mathcal{A}(C) \le \epsilon_1$.
\end{proposition}

Given Proposition~\ref{prop:bound}, Theorem~\ref{thm:limit} follows by an application of Gromov compactness.  We now sketch how this goes.

\begin{proof}[Sketch of proof of Theorem~\ref{thm:limit}, assuming Proposition~\ref{prop:bound}]

Here is a sketch of how this goes.  Fix a point $z$ in an element $S \in \mathcal{L}$.  It suffices to show that for small time, the trajectory through $z$ stays in $U$.  Look at a small ball around this point.  Then Proposition~\ref{prop:bound} implies that we are justified in applying the ``target-local" Gromov compactness of Fish \cite{fish}.  We therefore get convergence to a curve passing through $z$.  However, this curve has zero action and so must be invariant.

\end{proof}

It remains to explain the idea behind the proof of Proposition~\ref{prop:bound}.  The proof is rather technical.  

The statement of Proposition~\ref{prop:bound} is inspired by a local-area bound in the Fish-Hofer theory, and the proof builds on their estimates.
In the Fish-Hofer theory, they consider a single curve (the so-called ``feral" curve), with finite action but unbounded ``Hofer" energy; in our cases, we are considering a sequence of curves, with actions going to $0$ but Hofer energy that, while finite, is blowing up.  Perhaps the main complication with this is that in the Fish-Hofer setting, one can reduce to considering annular curves.  Here, we need to consider curves with much more complicated topology.   

\begin{proof}[Idea of the proof of Proposition~\ref{prop:bound}]

We will have to be very brief, and just give a flavor for the argument, referring the reader to \cite{cgp} for more details, or to the outline in \cite[Sec. 5.5]{cgp} for a briefer sketch.  

To start, one can assume that $C$ is contained in a region on the form $[a_0,a_1] \times Y$, where $a_0$ and $a_1$ are regular values of the height function; see [Sec. 5.5, paper].  

Now, in both our setting, and the Fish-Hofer setting, one starts by cutting the curve (after a slight perturbation, which we will ignore for expository simplicity) into small pieces called ``tracts".   A {\em tract} is a compact embedded surface with boundary and corners, such that the boundary decomposes into ``horizontal components", which are components that are level sets of the height function on the symplectization, and ``vertical components", which are components that are gradient flow lines of the height function on the symplectization.  These feature prominently in the Fish-Hofer setting: for example, they enjoy a crucial ``exponential area bound" which will allow us to bound the area of a particular tract from the length of its bottom horizontal components.

Now recall that Proposition~\ref{prop:bound} requires a bound on the area of the part of the curve in the radius $x$ ball.  The basic strategy, then, is to i) cut the curve into tracts, such that each boundary component is not too long; ii) bound the number of tracts covering the component of the curve in the radius $x$ ball; iii) bound the area of each such tract.    

The main ingredient for proving iii) from i) is the aforementioned exponential area bound: given this, we need only to bound the number of components of the boundary, and a combinatorial analysis of how the tracts can fit together to form $C$ accomplishes this by giving a bound in terms of $\chi(C)$.  

To prove ii) from i), the first point is that most of the tracts are in fact rectangular, i.e. have just two vertical boundary components; this follows from the same sorts of combinatorial topology arguments for deducing iii) from i).  Thus, we can restrict to rectangular tracts, and a key insight, taken from the Fish-Hofer theory, is that if $x$ is sufficiently small, the part of $C$ in the radius $x$ ball can not intersect both vertical boundary components of a single tract; this is proved by a triangle inequality argument.  Once we know that most tracts are rectangular, and the rectangular tracts do not have both vertical boundary components meeting the relevant parts of $C$, another combinatorial argument gives the desired bound.      

To prove i), recall that we can assume that $C$ is contained in $[a_0,a_1] \times Y$, and then we can define the ``lower" boundary $\partial^-_h C = \lbrace a_0 \rbrace Y \cap C$, defining the ``upper" boundary $\partial^+_h$ analogously.  Then the first point is that ``most" initial conditions on $\partial^-_h C$ have the property that the gradient flow of the height function meets $\partial^+_h C$: this parallels a fact in the Fish-Hofer theory and is proved similarly, so we will omit the proof.  Given this, to produce the desired tract decomposition, one chops up $\partial^-_h C$ into a finite number of pieces, each homeomorphic to a closed interval or a circle, with disjoint interiors and appropriate control on the lengths.  One now shows that each interval piece contains an initial condition with a flow line meeting $\partial^+_h C$; in fact, we can do better: we can arrange, by making use of an intricate lemma from the Fish-Hofer theory \cite{fh} that this flow line has a controlled length.  One now forms the desired tract decomposition by taking the connected components of the complements of these flow lines.        
 


\end{proof}

\subsection{Low action curves in low-dimensional symplectic dynamics}

In order to use Theorem~\ref{thm:limit}, we need a suitable sequence of curves $C_k$.  By the version of Theorem~\ref{thm:limit} mentioned above for parameters varying in a sequence, it suffices to find such curves for nondegenerate contact forms. We can find such a sequence using ideas related to those in the proof of Theorem~\ref{thm:hwz}.  Indeed, recall that we showed there that for a length $N$ $U$-sequence, the curves on average have action $O(1/\sqrt{N})$.  Thus, by taking larger and larger powers of $N$, we can similarly obtain an entire sequence of curves with actions tending to $0$.  

However, in order to apply Theorem~\ref{thm:limit}, we also need to know that the Euler characteristics of this sequence of curves does not tend to negative infinity.  In fact, we show that the Euler characteristic can be assumed bounded from below by $-2$:

\begin{proposition} \cite[Prop. 2.3]{cgp}
\label{prop:ECHcurves}
Let $\lambda$ be a nondegenerate contact form on a closed three-manifold $Y$ and assume that $\xi$ is torsion.  There always exists a sequence of $J$-holomorphic curves $C_k$ in $\mathbb{R} \times Y$ with actions tending to $0$ and Euler characteristic uniformly bounded.  In fact, the $C_k$ can be taken to pass through a marked point $(0,z)$ as long as $z$ is not on any Reeb orbit.

\end{proposition}

\begin{proof}[Idea of the proof]
For brevity we will just give a sketch, referring to \cite{cgp} for the details.

The idea for this builds on the ideas mentioned in \S\ref{sec:utowers}.  Recall that there, one related the ECH index to the $J_0$ index, via the formula \eqref{eqn:diff}, and showed that the two are approximately equal, under the assumption of finitely many simple Reeb orbits.  What is required to bring this to bear in the present situation is to drop the hypothesis of finitely many Reeb orbits.   As in \S\ref{sec:utowers}, there are two terms in \eqref{eqn:diff} that need to be dealt with: a Chern class term and a Conley-Zehnder index term.  

Let us first explain how to do this when the contact structure is trivial.  Here, one can take a global trivialization, so that the Chern class term vanishes, and then the Conley-Zehnder term can be handled by showing that the rotation of the linearized flow, suitably interpreted, is proportional to the length and also estimates the Conley-Zehnder index.  One then applies the Weyl Law, Theorem~\ref{thm:weyl}.

The torsion, but not trivial, case is harder.  We use the following trick.  The $n^{th}$ tensor power
$\xi \otimes \ldots \otimes \xi$ of the contact structure, for suitable $n$, will be trivial.  In general, the linearized flow on $\xi$ does not extend to $\xi \otimes \ldots \otimes \xi$.  However, one can take the unitary part of the linearized flow after polar decomposition.  One can show that this still bounds the Conley-Zehnder index.  On the other hand, the unitary part also extends to  $\xi \otimes \ldots \otimes \xi$.  Now, $I - J_0$, by \eqref{eqn:diff}, makes sense  for the bundle $\xi \otimes \ldots \otimes \xi$ even though it is not a contact structure.  Then the same argument in the case where the contact structure is trivial shows that this term  on the $n^{th}$ tensor power,
can be controlled by the Weyl law; on the other hand, its relation to the term on $\xi$ is just multiplication by $n$.  Thus we obtain the desired control.

\end{proof}

\subsection{Putting it all together}

Let us now explain how to combine the previous results to prove the Le Calvez - Yoccoz property.

To proceed, we need the following important observation:

\begin{lemma} \cite[Prop. 1.9]{cgp}
\label{lem:conn}
The limit set $\mathcal{L}$ is connected in the Hausdorff topology, after passing to a subsequence if necessary.  
\end{lemma}

The lemma is the motivation for taking slices of the curves of non-zero height.  In a first attempt, one might try to take slices at fixed values of the $\mathbb{R}$ co-ordinate in defining the limit set.  However, then the analogous limit set would not be connected.

The idea then is as follows.  Say that a closed non-trivial invariant set $S$ is {\em not isolated} if for every open set $U$ containing $S$, there is a point $p \in U \setminus S$ whose orbit closure stays in $S$.  (Otherwise, we say that $S$ is isolated.)  Then Theorem~\ref{thm:ley} follows from the following result of independent interest:

\begin{theorem}
\label{thm:nontriv}
Let $(Y,\lambda)$ be a closed three-manifold with contact form and assume that $\xi$ is torsion.  Any closed non-trivial invariant set containing the union of the periodic orbits is not isolated.
\end{theorem}  

The picture one should have given Theorem~\ref{thm:nontriv} is then that the invariant sets radiate outward from the periodic orbits.  

Here is the basic idea for how to prove the theorem, given Theorem~\ref{thm:limit}, Proposition~\ref{prop:ECHcurves} and Lemma~\ref{lem:conn}; for more details, see \cite[Sec. 2]{cgp}.

\begin{proof}[Idea of the proof.]

Let $\Lambda$ be a closed non-trivial invariant set containing the union of the periodic orbits.  Each individual curve $C_k$ is asymptotic to periodic orbits.  So, if we look at height $2$ slices with sufficiently large $s$, these are very close to periodic orbits.   Thus, some element of the limit set projects to a subset of the closure of the set of periodic orbits.  On the other hand, the limit set is connected; and, we can place the marked point $(0,z)$ anywhere not on a periodic orbit.  In this way, we obtain elements of the projection of the limit set to $Y$ arbitrarily close to $\Lambda$.   
\end{proof}


\section{Bonus: Seiberg-Witten equations and Weyl laws}
\label{sec:weyl}

We now explain a little bit\footnote{During the lectures, we did not have time to go into this material 
at all, though there was a question about it during the question and answer session.} of the story behind the two Weyl laws that figured prominently in the proofs of our two main theorems.  (Recall that to prove the Simplicity Conjecture, one just needed a special case of the surface Weyl law, which as we discussed can be established by direct computation; however, to prove the Closing Lemma, we do need the more general statement that we discuss here.)

\subsection{The contact case}

We start by explaining the main ideas behind the proof of Theorem~\ref{thm:weyl}. 

\subsubsection*{Preliminaries}

The proof uses the Seiberg-Witten invariants.   In our case, we are interested in applying the Seiberg-Witten equations to better understand $(Y,\lambda)$. More precisely, we will be interested in the following system of equations:
 \begin{equation}
 \label{eqn:3dsw}
 \star F_A = r( \langle \rho(\cdot) \psi, \psi \rangle - i \lambda) + i \bar{\omega}, \quad \quad D_A \psi = 0.
\end{equation}
These are equations for a pair $(A,\psi)$, which we call a {\em configuration}.  Let us explain the notation.  (We will be a bit brief, referring to \cite{ECHasymptotics} for more detail.)  Here, $A$ is a Hermitian connection on a Hermitian line bundle $E$ and $F_A$ denotes its curvature.  The line bundle $E$ induces a spin-c structure as follows:
after a choice of compatible metric, the contact form $\lambda$ determines a spin-c structure 
\[ s_{\xi} = \mathbb{C} \oplus \xi\]
and then $E$ determines a spin-c structure via tensor product with $s_{\xi}$.  The Hermitian connection $A$ then induces a spin-c connection, and $D_A$ denotes the Dirac operator associated to this spin-c connected; $\psi$ is a section of this spin-c structure and $\rho$ denotes Clifford multiplication.  Finally, $r \ge 1$ is a positive real number, and $\bar{\omega}$ is a harmonic $1$-form whose curvature represents $\pi \star c_1(\xi)$.  (One also might need to take a further small perturbation of the equations to obtain transversality, but we have omitted this for expository simplicity.)

The reason we are interested in \eqref{eqn:3dsw} is as follows.  There is an {\em energy} of a configuration
\begin{equation}
\label{eqn:energy}
E(A,\Psi) := i\int_Y \lambda \wedge F_A 
\end{equation}
For large $r$, it turns out that solutions of bounded energy concentrate along sets of Reeb orbits.  In fact, we can do better:

\begin{prop}
\label{prop:echfromsw}
Let $\sigma$ be a nonzero class in $ECH(Y,\lambda)$ and let $c_{\sigma}(\lambda)$ denote the corresponding spectral invariant.  Then there exists a family of solutions $(A(r),\psi(r))$ to \eqref{eqn:3dsw} such that
\[ \lim_{r \to \infty} \frac{E(A(r))}{2 \pi} = c_{\sigma}(\lambda).\]
\end{prop} 

Proposition~\ref{prop:echfromsw} is a kind of quantitative refinement of a piece of Taubes' isomorphism \cite{Taubes-I}.  We will say a few words about the proof below and first discuss its significance.  To use Proposition~\ref{prop:echfromsw}, the basic idea is as follows:
\begin{enumerate}[(a)]
\item Show that for some small enough $r_0$, we can determine $E(A(r_0),\psi(r_0))$ up to some small error.
\item Show that $E$ does not change too much as $r$ varies.
\end{enumerate}
To implement this, a basic challenge is that $E(A(r),\psi(r))$ can not in general be assumed continuous as $r$ varies.  To get around this, it is helpful to introduce the {\em perturbed Chern-Simons-Dirac functional} 
\[ F(A,\Psi) := \frac{1}{2}( cs(A) - r E(A)) + r \int_Y \langle D_A \psi, \psi \rangle ,\]
where $cs$ denotes the {\em Chern-Simons} functional 
\[ cs(A): = - \int_Y (A - A_E) \wedge (F_A + F_{A_E} - 2 i \star \bar{\omega}),\]
which requires a choice of reference connection $A_E$ on $E$.   To motivate the definition of $F$, we remark that solutions to \eqref{eqn:3dsw} are critical points of $F$.   In fact, it is quite useful in what follows to think of $F$ as a one-parameter family of functions, depending on the parameter $r$.

We then have the following refinement of Proposition~\ref{prop:echfromsw}:

\begin{prop}
\label{prop:piecewise}
We can assume that the solutions $(A(r),\psi(r))$ in Proposition~\ref{prop:echfromsw} have the property that $ f(r): = F(A(r),\psi(r))$ is piecewise smooth and continuous.  
\end{prop} 

\begin{proof}[Idea of the proofs of Proposition~\ref{prop:echfromsw} and Proposition~\ref{prop:piecewise} ] 

The proofs build on Taubes' work establishing the isomorphism \eqref{eqn:echsw}.   The basic idea is to regard the Seiberg-Witten Floer cohomology as the formal Morse homology of $F$, and try to argue in analogy with the finite dimensional case.  

In the finite dimensional case, if we are given a one parameter family of functions that are Morse except at a discrete set of values, then a choice of nonzero class in the Morse homology gives rise to a one parameter family of critical points, via a minimax construction: at each parameter value $r$, we look at the minimum value of the Morse function required to represent the class, and take a critical point realizing this value; then, one can show that this gives rise to a family of critical 
values
that are piecewise smooth and continuous. 

In our case, the class $\sigma$ gives rise to a Seiberg-Witten class via the isomorphism \eqref{eqn:echsw}.  Roughly speaking, we then mimic the construction from the previous paragraph to get the family of solutions $(A(r),\psi(r))$; this will have the desired properties for $f$.  The work of Taubes (see e.g. \cite{Taubes-I}) then implies that (after applying a gauge transformation if necessary) the curvature of $A(r)$ concentrates along an orbit set as $r \to \infty$, with action approximately $E/2\pi$.   Some further fiddling with the filtrations in Taubes' isomorphism shows that in fact the action of this orbit set represents the spectral invariant; see \cite[Prop. 2.6]{ECHasymptotics}.
\end{proof}

\subsubsection*{The grading and the Chern-Simons functional}

Before proceeding, it is also very useful to note that  solutions to \eqref{eqn:3dsw} have a grading $gr$, defined by spectral flow to some fixed connection.   The following crucial estimate due to Taubes bounds this in terms of the Chern-Simons functional:

\begin{prop} [Taubes]
\label{prop:gradchsimons}
There is a constant $C$ such that 
\[ |gr(A,\psi) + \frac{1}{4 \pi^2} cs(A,\psi)| \le C r^{31/16}\]
for any solution $(A,\psi)$ to \eqref{eqn:3dsw}.
\end{prop} 

The proof of Proposition~\ref{prop:gradchsimons} is beyond the scope of these notes.  For one treatment, giving an improved bound on the remainder term, see \cite{cgsavale}: the basic idea in this argument is to apply the Atiyah-Patodi-Singer index theorem to the compute the grading, after which one finds that that Atiyah-Singer integrand can be identified with $cs$; the bound on the difference then arises from estimating the ``$\eta$-invariant in the APS theorem, via a heat kernel expansion.

\subsubsection*{The key computation}

Let us now try to implement the scheme from the previous section, i.e. elaborate on items (a) and (b) there.  Recall that these were as follows.

\vspace{1 mm} {\em (a): Show that for some small enough $r_0$, we can determine $E(A(r_0),\psi(r_0))$ up to some small error.}

\begin{proof}[Idea of the proof]
The idea for this is to use the fact that \eqref{eqn:3dsw} admits reducible solutions.  These are solutions with $\psi = 0$.  They can be written down rather explicitly.   Choose some connection $A_0$ with $F_{A_0} = i \star \bar{\omega}$.  Then it is easy to check that
\begin{equation}
\label{eqn:red}
(A,\psi) = (A_0 - \frac{1}{2} i r \lambda, 0 )
\end{equation}
is a solution to \eqref{eqn:3dsw}.  Every other reducible solution differs from this one by the addition of a closed one-form.    Plugging in, one therefore gets that the energy of a reducible solution is approximately $\frac{r}{2} vol(Y,\lambda)$.  Note the appearance of the volume term!     

To start to organize this into a useful form, it is crucial to note that as $r$ varies, the grading of the reducible solutions to \eqref{eqn:red} increases.  Indeed, in view of Proposition~\ref{prop:gradchsimons}, one can get a good estimate by applying \eqref{eqn:red},  More precisely, one can compute that for any solution of \eqref{eqn:red}, 
\[ cs(A) = \frac{1}{4} r^2 vol(Y,\lambda) + O(r).\]

Ignoring lower order terms in $r$, we therefore see that for reducibles the energy $E$ and the grading $gr$ satisfy
\[ E^2 = vol(Y,\lambda) gr ,\]
which is exactly the relation we want in the Weyl law.

This motivates the choice of $r_0$.   As we will see, we will choose for each relevant grading $j$ a  different value of $r_0$.  That is, for each relevant grading $j$, we should look at a value of $r_j$ such that for the family $(A(r),\psi(r))$ in Proposition~\ref{prop:piecewise} at $r = r_j$ is reducible.   Then, a successful implementation of (b) would establish the theorem.   Choosing such an $r_j$ is the basic idea behind (a).
\end{proof}

\vspace{1 mm} {\em (b): Show that $E$ does not change too much as $r$ varies..}

\begin{proof}[Idea of the proof]

The idea for this builds on Taubes' proof of the Weinstein conjecture.  The functional $E$ itself does not have particularly good continuity properties as $r$ varies.  However, as asserted by Proposition~\ref{prop:piecewise}, the functional $F$ does.  Moreover, the Dirac term in $F$ vanishes for solutions, so we can write
\[ F = \frac{1}{2}(cs - r E).\]
The idea then is to estimate $E$  by estimating $F$.  In other words, we have
\[E = - 2 F/r + cs/2r .\]
One can in addition show that for a family as in Proposition~\ref{prop:piecewise}, the $cs$ term is $o(r)$ as $r$ tends to $+\infty$, see e.g. the survey of Taubes' work in \cite{tw}.  Thus, if we define
\[ v = - 2 F/r,\]
then $v$ and $E$ have the same limit as $r \to \infty$.  We therefore want to estimate $v$.

To estimate $v$, the idea is to compute that
\[ v' = cs/r^2.\]
Thus, the estimate Proposition~\ref{prop:gradchsimons} gives a certain amount of control over the size of $v'$.  A quick examination of what kind of estimate this would give shows that with that estimate, alone, the change in $v$ as $r$ varies from $r_j$ to $\infty$ could in fact be quite large.  However, as we said above, $cs$ is actually $o(r)$ as $r \to \infty$, which is a much stronger estimate.  The basic idea, then, is to estimate the change in $v$ in two stages: one first isolates a value $r_j^*$ of $r$ after which $cs$ has a $o(r)$ bound --- from here, $v$ does not change much at all --- and then estimates the change in $v$ from $r_j$ to $r^*_j$, using the estimate from Proposition~\ref{prop:gradchsimons}.

\end{proof}

\subsection{Some remarks to give a sense for the mapping torus case}

The Weyl law in the PFH case is more complicated.  We will only say a few words.  Our goal is to explain the crucial differences from the ECH case, and give some idea about how to overcome them.

Since, as we explained, PFH is also isomorphic to a version of Seiberg-Witten Floer cohomology, it is natural to attempt to prove the Weyl law using Seiberg-Witten equations.  This can indeed be made to work, but there are several challenges that need to be overcome.  Here are the main ones:

\begin{enumerate}[(a)]
\item ECH is known to always be non-trivial trivial.  However, PFH can certainly vanish.  For example, an irrational translation of a two-torus has no periodic points at all, and so the PFH vanishes in every positive degree.
\item As we saw above, in the proof of the contact case, reducible solutions played a key role.  However, the version of the Seiberg-Witten equations studied by Lee-Taubes, in establishing their isomorphism, admit no reducible solutions at all.
\item As we saw above, in the proof of the contact case, the energy played a key role.  However, in the PFH case, the most obviously analogous functional $A \to \int_Y dt \wedge F_A$ depends only on the spin-c structure, and so does not vary as $A$ varies with $r$.   
\item In the ECH case, we fix a homology class in $H_1(Y)$, and this corresponds to a single spin-c structure.  In the PFH case, we need to consider a sequence of degrees, and these correspond to a sequence of spin-c structures.  All estimates that we use therefore need to isolate any dependence on the choice of spin-c structure.
\item Continuing in the spirit of the previous item: in the ECH case, one needs the estimate Proposition~\ref{prop:gradchsimons}.  In that estimate, $C$ is a constant depending on a fixed spin-c structure.  In the PFH case, one needs to consider a sequence of spin-c structures and prove analogous estimate, understanding the role of $C$.  
\item Quantitative ideas play a key role in isomorphism between ECH and Seiberg-Witten.  However, in the PFH case, these kind of considerations are not particularly necessary (essentially because $dt$ is closed, in contrast to $\lambda$.)  Thus, a lot of this structure needs to be built up.
\end{enumerate}  

We resolve all of these issues in \cite{cgpz}.  Let us try to explain a bit about how each point is resolved.  

Concerning (a), we prove a new nontriviality result: we show that when the cohomology class of the canonical two-form on the mapping torus is rational, the PFH is always non-trivial in infinitely many degrees.  This is proved by using the isomorphism with Seiberg-Witten Floer cohomology.  On the Seiberg-Witten side, one needs to prove the corresponding non-triviality result about a version of the theory, where one takes nonexact perturbations.  
There is a variant of the Seiberg-Witten cohomology generated entirely by reducible solutions and by known properties from \cite{KMbook}, it suffices to prove the needed nontriviality result for this reducible variant.  This reducible variant in turn can be computed solely in terms of classical topology, and some wading through some spectral sequences gives the desired result.  

As for (b), we formulate a variant of the equations that admit reducible solutions and still give the same Floer cohomology.  This introduces a new complication that has no analog in the contact case: how to show that these reducibles contribute at some point to the spectral invariant; the argument involves considerations of properties of the ``completed" version of the theory.  

Concerning (c), the key point is that after a choice of reference connection $A_0$, the data encoded in the energy in the contact case can be thought of instead in terms of the integral $\int_Y d \lambda \wedge a$, where $A = A_0 + a$.  This formulation of the energy has an analog in the PFH case that is suitable for our purposes -- though, we do pay a price in that the choice of base connection introduces some further complications, requiring us to keep track of this non-canonical, particular when we work with sequences of spin-c structures, for which one has to make a whole sequence of choices; nevertheless, these issues can be overcome.  

As for (d) - (f), this requires new estimates and forms the content of \cite[Sec. 6]{cgpz}; for brevity, we will not say much more about how these estimates work, referring the reader instead to \cite{cgpz}.













\section{Other developments}

Let us now highlight a few other important developments in the area of low-dimensional symplectic dynamics over the last few years.    It is of course not possible to survey all the exciting recent progress in the area, so we focus on a few topics that the author finds particularly interesting.  


\subsubsection*{Global surfaces of section}

Recall the notion of a global surface of section.  This allows one to reduce to considering two-dimensional dynamics and it played a key role in the proof of Theorem~\ref{thm:hwz}.  

Remarkably, such global surfaces of section seem to be a central feature of three-dimensional Reeb flows:

\begin{theorem}\cite{CDHR,cm}
\label{thm:gss}
Let $(Y,\lambda)$ be a closed three-manifold with a contact form.  For generic choice of $\lambda$, the Reeb flow always has a global surface of section.  
\end{theorem}

The paper \cite{cm} also shows this for a generic Riemannian metrics; the argument in \cite{CDHR} can be modified \cite{hprivate}, without much difficulty, to prove this statement as well.

In both of these papers, the ECH $U$-map curves are a key input, though in the arguments one does not need the kind of control on the genus or the action as in the proof of Theorem~\ref{thm:hwz}.  A starting point for both works is that there is a $U$-map curve through any point $(0,z)$ in $\mathbb{R} \times Y$.  (This is itself a rather remarkable property of four-dimensional symplectizations, that one could call ``uniruledness" in analogy with similar properties in algebraic geometry and the symplectic geometry of closed manifolds.)  Then, there is a kind of surgery procedure building on work of Fried for assembling the curves into a GSS, possibly with quite large genus, from the projections of the $U$-map curves.

As a very interesting corollary, the authors find that a generic Reeb flow on a closed three-manifold has positive topological entropy.  In fact, they show that set of contact forms whose associated Reeb flow has positive topological entropy is open and dense.   For more about entropy, see our discussion below.  

In a different direction, it is also interesting to understand more about what kind of orbits bound global surfaces of section.  A longstanding question \cite{convex} asks whether for strictly convex domains (i.e. the boundary is a compact regular energy level of a smooth Hamiltonian with positive definite Hessian) in $\mathbb{R}^4$, the periodic orbit of shortest period bounds a global surface of section.  A proof establishing an affirmative answer to this question was recently announced \cite{olivertalk}.  This has an interesting consequence for the theory of ``symplectic capacities."  Symplectic capacities are measurements of symplectic size.  An example is the ``Gromov width", which is the size of the largest symplectically embedded ball.  It is interesting to ask to what degree they are unique, after some kind of normalization.  As explained in \cite{olivertalk}, the fact that the periodic orbit of shortest period bounds a global surface of section, for strictly convex domains, combined with a previous work \cite{egafa}, implies that for such domains there are at most two different normalized symplectic capacities that we know of.  In fact, another recent result shows that these two must be different; we discuss this further in our discussion of Viterbo's conjecture below. The papers \cite{egafa, hhr} prove related interesting results, relating the area of the smallest global surface of section to other symplectic capacities.  Additionally, the paper \cite{hhr} requires only a weaker notion of ``dynamical convexity"; we return to this in our discussion of the role of convexity in symplectic geometry below. 

\subsubsection*{Global systems of surfaces of section and two or infinity beyond the torsion case}

In Theorem~\ref{thm:gss}, one requires genericity on $\lambda$.  If one only asks that $\lambda$ is nondegenerate (but imposes no further conditions) then we do not know that there is a global surface of section, but \cite{CDR} proves the result that there is a related notion called a {\em broken book decomposition}: roughly speaking, this is like a system of surfaces of section, where it might not be the case that every flow line hits a single surface, but the flow maps one surface (a ``page") to another.  In fact, this result predates and was an input to the results \cite{CDHR,cm} on the existence of a GSS above; it also takes as a starting point the existence of a $U$-map through every point.  Using this broken book decomposition, one can also extend the two or infinity result to all nondegenerate contact forms:

\begin{theorem}\cite{CDR}
\label{thm:gss}
Let $(Y,\lambda)$ be a closed connected three-manifold with a nondegenerate contact form.  Then there are always two or infinitely many simple Reeb orbits.
\end{theorem}

It is also interesting to try to prove results like Theorem~\ref{thm:gss}, beyond the contact case.  For example, there is a generalization of a contact form called a ``stable Hamiltonian structure"; this has an associated vector field which includes both mapping torii and contact Reeb flows.  The paper \cite{stabham} classifies all non-degenerate stable Hamiltonian Reeb flows with finitely many simple periodic orbits.

\subsubsection*{Stability}

Let us now explain a beautiful application of three-dimensional contact geometry to geodesic flows.  A starting point for the arguments used to prove it are the broken book decompositions mentioned above.

Recall that a diffeomorphism $\psi$ is called $C^k$-{\em stable} if there is an open neighborhood of $\psi$ in the $C^k$-topology, such that any $h$ in this neighborhood is conjugate to $\psi$ by homeomorphisms.   There is an analogous notion of stability for geodesic flows.  It is natural to ask what kind of flows are stable and there is a celebrated conjecture of Palis and Smale about this: conjecturally, the $C^k$-stable flows are ``Axiom A".  For $k = 1$ this is known by famous work of Ma\~{n}e, but in higher regularity seems wide open.  We will not review the definition of Axiom A here, but in the geodesic context it is equivalent to the flow being Anosov, meaning that it is hyperbolic.  The following result confirms the Palis-Smale conjecture for $k = 2$ on surfaces:

\begin{theorem}\cite{cm2}
The $C^2$-structurally stable geodesic flows are precisely the Anosov flows.
\end{theorem}

\subsubsection*{Elementary invariants and dynamical interpretations of spectral invariants}

Another very interesting development involves defining spectral invariants like the PFH ones, but without using Floer theory in their definition.  Let us illustrate with an example.  As we explained, in our proof of the Simplicity Conjecture, the PFH spectral invariants play a starring role.  It turns out that one can define similar invariants, called  {\em elementary PFH invariants}, without using any kind of Floer theory; see \cite{e}.  These elementary invariants are just defined in terms of pseudholomorphic curves: roughly speaking, they measure the amount of area required so that there are always $J$-holomorphic curves of no bigger area through various collections of points, no matter where one puts the points and varies $J$.  The idea of replacing spectral invariants defined via Floer theory with these kind of optimizations over classes of pseudoholomorphic curves originates with work of McDuff-Seigel in the context of a different kind of Floer theory called symplectic field theory.    
It is shown in \cite{e} that these elementary PFH spectral invariants satisfy the same axioms as the usual ones, and so also suffice for proving the Simplicity Conjecture.  Thus, one can prove the Simplicity Conjecture without using Floer theoretic ideas at all.

 As the technical details of Floer theory are quite substantial, the elementary invariants streamline many constructions.  However, in addition, they sometimes give results that stronger than what is currently known via the Floer theory methods.  For example, in the context of closing lemmas, one can prove a quantitative refinement, giving a bound on the first time in which an orbit appears in terms of the ambient data; see \cite{eh}.  On the other hand, currently they can not recover certain important results proved via Floer theory, for example the fully general $C^{\infty}$ closing lemma for area-preserving diffeomorphisms.
 
In a different direction, an interesting project, initiated in \cite{hs}, seeks to give purely topological and dynamical formulations of spectral invariants for Hamiltonians on surfaces.  A recent result \cite{dustin} gives exciting progress on this program and some further related results.


\subsubsection*{Feral curves, invariant sets and quantitative almost existence}
 
As we briefly mentioned above, the groundbreaking work of Fish-Hofer \cite{fh, fh3} introduces a new class of pseudoholomorphic curve, the ``feral" cuves, for which one can prove meaningful compactness statements without requiring a ``Hofer energy" bound.  Fish-Hofer use these curves to show that a compact energy hypersurface of a Hamitonian on $\mathbb{R}^4$ is never minimal (i.e. has a nontrivial invariant set), answering the four-dimensional case of an old question due to Herman \cite{herman}.  As was previously mentioned, Prasad has further developed the theory (e.g. in \cite{prasadnew}) and has recently shown that in fact such an energy level has infinitely many proper compact invariant sets whose union is dense in the manifold; ideas from feral curve theory play a key role in this work as well.
The theory of feral curves also plays a key role in the discussion of the Le Calvez - Yoccoz property in \ref{sec:ley}.  

To explain a different sort of result related to the theory, note first that an old result shows that on $\mathbb{R}^{2n}$, when the Hamiltonian is proper and smooth, almost every energy level has a periodic orbit.  The article \cite{fh2} gives a proof of this result using the new feral curve theory.   A very recent work \cite{prasadnew} shows that when $n = 2$, in fact one can improve this to show the existence of at least two simple periodic orbits on almost every energy level.     The same work also shows that in the $n = 2$ case, almost every energy hypersurface satisfies the Le Calvez - Yoccoz property.   


All of these applications are low-dimensional, but we should stress that the compactness theory underlying the applications above have no such dimensional restriction.

\subsubsection*{Entropy}

Recent years have also seen considerable growth in our understanding of entropy in low-dimensional conservative systems, via Floer theoretic or pseudoholomorphic curve techniques.  

One series of interesting works have focused on recovering entropy via Floer theory.    For example, in \cite{entropy1} the authors define the {\em barcode entropy} of a Hamiltonian diffeomorphism, a quantity extracted from its Floer theory.  For closed surfaces, they show that this agrees with the topological entropy.  A later paper \cite{entropy3} extends this to Reeb flows on boundaries of any four-dimensional Liouville domain.  Another paper \cite{entropy2} proves this for geodesic flows on surfaces.   On the other hand, \cite{cineli} gives examples in dimension $\ge 6$ of Hamiltonian diffeomorphisms with zero barcode entropy and positive topological entropy; hence, low-dimensionality seems to be a crucial feature in these works.  Indeed, in low-dimensions, topological entropy is closely connected to the growth rate of periodic orbits, but in higher dimensions there are examples of maps with positive topological entropy and no periodic orbits at all; see e.g. the discussion in \cite{cineli}.

In a different direction, \cite{am} proves the interesting result that for surfaces, the topological entropy is lower semi-continuous with respect to Hofer's metric for Hamiltonan diffeomorphisms; \cite{hutchingsentropy} then proves a version of this for some area-preserving diffeomorphisms that are not necessarily Hamiltonian.  These are very subtle results, as Hofer closeness does not imply $C^0$-closeness.  A similar result was then proved for Reeb flows on closed three-manifolds \cite{entropyreeb}; this is again quite subtle, since a $C^0$-perturbation of the contact form can have a very strong effect on the dynamics of the associated Reeb vector field, which is defined via a derivative of the contact form.    
A related direction involves structures that ``force" positive topological entropy.  For example, \cite{al1} shows that if a Reeb flow on a contact $3$-manifold has a link of Reeb orbits, transverse to the contact structure and satisfying certain Floer theoretic conditions, then the flow must have positive topological entropy.  
Another work \cite{al2} shows that if a Reeb flow on a three-manifold is Anosov, then every Reeb flow associated to the same contact structure has positive topological entropy.

Although a little bit beyond the (very recent) time window we are focusing on in these notes, we feel based on the importance of the result that we should also mention a classic work of Bramham \cite{br1} here, regarding an old question of Katok about conservative zero entropy low-dimensional systems.  This question of Katok asks whether such systems are limits of integrable ones.  Specializing now to the case of the two-dimensional disc (with its standard area form), a very important class of such systems are ``pseudorotations", namely systems with exactly one periodic point.  All known ergodic disc maps with zero entropy are pseudorotations.  The main result of \cite{br1} is that pseudorotations of the two-disc are $C^0$-limits of conjugates of rotations.  In fact, for two-dimensional annuli one can analogously define pseudorotations, and a more recent result \cite{bramz} shows that such pseudorotations 
are $C^{\infty}$ limits of conjugates of rotations, when the pseudorotation has a generic rotation number.

We discuss some more recent results about 
pseudorotations
next.

\subsubsection*{Pseudorotations}

We introduced this important class of dynamical system above and explained their significance in connection with conservative zero entropy low-dimensional dynamical systems.   For a beautiful survey about the topic, see \cite{pseudosurvey}.  We mention here just a few results.

For maps of the (closed) two-dimensional disk, a pseudorotation by definition has just one periodic point, which is a fixed point, and it is interesting to understand what kind of information is carried by this fixed point.  A work by Hutchings \cite{calabipseudo}, under some hypotheses that were removed in \cite{abror}, shows that for a pseudorotation the action of this fixed point is equal to the Calabi invariant discussed above.  It is also shown in \cite{abror} that the Calabi invariant agrees with the boundary rotation number.  Another interesting quantity for such disk maps is the ``asymptotic mean action", defined for (almost every) point $x$ in the disc as the Birkhoff average of the average action of the orbit through $x$.  It is shown in \cite{abrorlec} that for pseudorotations, the asymptotic mean action agrees with the boundary rotation number.    Hutchings' proof in \cite{calabipseudo} uses embedded contact homology, but a later work \cite{calabipseudocalvez} gives another proof using more dynamical techniques.  

An old conjecture due to Birkhoff, foundational to the subject, asks about analytic pseudorotations, specifically whether such pseudorotations have to be conjugate to true rotations, in the case of the two-sphere or the two-dimensional cylinder.  A famous construction due to Anosov-Katok \cite{anosovkatok}, published in 1970, allows one to construct smooth counterexamples to this conjecture, as well as smooth pseudorotations of the closed two-dimensional disk that are not conjugate to true rotations.    
In fact, the existence of analytic counterexamples to Birkhoff's conjectures by a technique that builds on the Anosov-Katok method, and a solution to some related problems, have recently been presented in \cite{b1,b2}.  These are constructive works, and so methodologically very far from the bulk of what is discussed in these notes.  Returning now to the case of the disk, it is interesting to understand to what degree the situation changes if one assumes in addition some knowledge about the behavior on the behavior.  In particular, a famous conjecture due to Herman \cite{herman}
asserts that a pseudorotation is in fact conjugate to a rotation, if the rotation number on the boundary is Diophantine.  On the other hand, it is known that every Liouvillean number is the rotation number of a weakly mixing pseudorotation \cite{fayadk}.


It is an old question whether a zero entropy conservative map of the two-disk can be strongly mixing.  
An important work of Bramham \cite{br2} (slightly older than the timeframe we are focusing on here) shows that a pseudorotation of the two-disk can not be strongly mixing, as long as the boundary rotation number lies in a certain dense set of Liouville numbers.   The aforementioned \cite{bramz} shows that for pseudo-rotations of the two-dimensional annulus, if the rotation number lies in a generic set, then $C^{\infty}$-generically the pseudorotation 
is
weakly mixing.  





For Reeb flows on closed three-manifolds, one can define a pseudorotation to be a flow with exactly two periodic orbits.   (Two is the minimal number of periodic orbits by \cite{CGH}.)  It is shown in \cite{CGHHL} that such flows have a close relationship to two-dimensional pseudorotations: such a flow always admits a global surface of section, with the topology of an (open) disc.  As the methods for proving this are related to many of the themes of these notes, let us take a moment to give a sense for them.  A crucial point is that the existence of exactly two periodic orbits forces the contact form to be nondegenerate; this is proved through an application of the Weyl law Theorem~\ref{thm:weyl}, after relating the rotation numbers of these orbits to the asymptotics of the ECH grading.     
Once one knows the contact form is nondegenerate, one can bring ECH to bear in a very strong form, and the same work shows a number of results relating these periodic orbits to information about the flow, by again using the Weyl law; for example, in the case of the three-sphere, it is shown that the product of the periods of the orbits is equal to the volume of the manifold; these kind of relations are used to show that the form is dynamically convex, which in turn is used to deduce the existence of the global surface of section.  Going in the other direction, it is shown in \cite{pseudoreeb} through completely different, constructive, techniques that many pseudorotations of the two-disk can be suspended to Reeb flows.    

Pseudorotations are interesting in higher dimensions as well, though that is not the focus of these notes.  For more, we refer the reader to the foundational paper \cite{gginvent} in this direction.  The authors prove, for example, a variant of Bramham's $C^0$-rigidity theorem for Hamiltonian diffeomorphisms of $\mathbb{C}P^{n}$ with exactly $n+1$ periodic points.

\subsubsection*{Systolic phenomena and Viterbo's conjecture}

Let us now say a few words about a beautiful story that is not actually unique to low-dimensions, but in which some insights in low-dimensions played a key role. 

In situations where periodic orbits are known to exist, it is also interesting to ask about the period $T_{min}$ of the shortest periodic orbit.  For Reeb flows, one can always scale the contact form by a constant to change the period, but it is interesting to search for bounds on the ``systolic ratio" $T_{min}/\sqrt{vol(Y,\lambda)}$.  (There is an analogous scale-invariant systolic ratio defined in any dimension.)  For example, the longstanding ``Viterbo" conjecture asserts that for convex domains, the systolic ratio is maximized for the standard ball (more precisely, one requires the interior to be symplectomorphic to the standard ball).  
This is known to imply the famous Mahler conjecture from convex geometry \cite{mahler}.

 A foundational result \cite{systolic} from 2015 shows that without convexity, Viterbo's conjecture is not true: in fact, \cite{systolic} shows that for star-shaped domains the systolic ratio can be arbitrarily large.  The same work also shows that, on the other hand, Viterbo's conjecture does hold in a neighborhood of the standard ball, or more generally in the neighborhood of any Reeb flow with all orbits periodic.  The paper \cite{systolic} is about dynamics on the three-sphere,
but a later series of works \cite{a1, a2, a3, a5, a6,b} prove more general results, and in arbitrary dimension.   One such more general result involves the uniqueness question for symplectic capacities introduced above: it is shown that, $C^2$ close to the standard ball, all normalized capacities do agree.  

Returning to Viterbo's conjecture itself, a very recent result shows that the conjecture is in fact false! See \cite{viterbocounter}.   The construction of a counter example involves a convex domain in $\mathbb{R}^4$, whose flow encodes a kind of billiard dynamics, called Minkowski billiards.  The authors explain how to extract counter examples in any dimension from this.  

For more about systolic phenomena, see the beautiful survey \cite{benedetti} and the references therein.

\subsubsection*{The symplectic meaning of convexity}

Here is another example of a beautiful story for which some key insights first occurred in low-dimensions.

In various developments mentioned above, convexity plays a key role.  However, its symplectic meaning still remains unclear in some ways.  
One might hope for a more symplectic replacement.  For example, an old result \cite{convex} implies that smooth convex domains enjoy a ``dynamical convexity" for the Reeb flow on the boundary; in the case of domains in $\mathbb{R}^4$, this is equivalent to requiring that the Conley-Zehnder index from above is at least $3$ for the global trivialization.  It had been a longstanding question whether the converse holds: does dynamical convexity, say for a star-shaped domain in $\mathbb{R}^4$, imply that the interior is symplectomorphic to something convex?  The recent work \cite{chaidezedtmair} establishes a subtle inequality for convex domains, which implies in particular that dynamical convex is not equivalent to convexity.  Later, \cite{chaidezedtmair2} the authors prove an analogous result in any dimension.  Returning to dimension $4$, the more recent works \cite{cghind,ramosetal} give different proofs that dynamical convexity and convexity are not the same, using the theory of ECH capacities; ECH capacities are closely related to the ECH spectral invariants that we discussed previously in these notes.  In all of these works, a crucial input from convexity is the existence of ``John's ellipsoid": this guarantees that any convex domain can be sandwiched between an ellipsoid and a rescaled copy of this ellipsoid, with the volume of all three of these objects close together.   The works show, through various techniques, that this forces a certain amount of rigidity.







\subsubsection*{More about the structure of homeomorphism groups}

We also have a better understanding of the structure of $\Homeo_c(D^2,\mu_{std})$ and its companion group $\Homeo_c(S^2,\mu_{std})$.  Theorem~\ref{thm:simpconjecture} states that these groups are not simple, but one would like to know more about their structure.   One would also like to know the situation for surfaces with genus.

The paper \cite{cgetal} proves an analogue of Theorem~\ref{thm:simpconjecture} for any compact surface.

\begin{theorem} \cite{cgetal}
Let $(S,\mu)$ be a compact surface with suitable measure.  The kernel of Fathi's mass-flow homomorphism is not simple.
\end{theorem}

The proof of this theorem also used a different kind of spectral invariant called ``link spectral invariants"  \cite{cgetal, ps}; in the context of \cite{cgetal}, this is defined via a quantitative version of a kind of variant of Heegaard Floer homology.

Here are some other results in these works.  The paper \cite{cgetal} constructs quasimorphisms on $\Homeo_c(S^2,\mu_{std})$; it follows from this that the commutator length in these groups is unbounded.   In \cite{cgetal2} some new normal subgroups  are constructed in the case of genus $0$ surfaces; for example, a normal subgroup is constructed via demanding $O(1)$ subleading asymptotics in an analogous Weyl law.     The paper \cite{cgetal2} also explains how to extend the Calabi homomorphism (non-canonically) to the full group in the genus 
$0$ case, answering one of Fathi's questions that we mentioned in the introduction.  The paper \cite{sobhan1} generalizes these results from \cite{cgetal2} to arbitrary genus.  

\subsubsection*{Hofer geometry}

The PFH spectral invariants can also be used to understand questions about the geometry of Hofer's metric; it turns out that many basic questions about the large-scale geometry are not known.   For example, in the case of $S^2$, an old question had asked whether or not the group of Hamiltonian diffeomorphisms is quasi-isometric to $\mathbb{R}$.  It is proved in \cite{simp2}, using PFH spectral invariants, that, in fact, this space admits quasi-isometric embeddings of $\mathbb{R}^N$, for any $N$.  This was proved simultaneously and independently (in fact, the stronger statement that embeddings of $C^{\infty}_c(0,1)$ exist was proved), with a different kind of invariant, in \cite{ps}.   One can also define a variant of Hofer's metric on braid groups, and  the aforementioned link spectral invariants can be used to obtain interesting results about this Hofer geometry of braids on surfaces \cite{sobhan2, sobhan3}.

\subsubsection*{The group of area-preserving diffeomorphisms of an open disc}

There are also interesting questions about subgroups of homeomorphism groups of the disc, beyond the compactly supported case.  For example, fix a volume form $\mu$ on the open $n$-disc and consider the group of diffeomorphisms preserving $\mu$.  When $n \ge 3$, it was shown by McDuff \cite{mcduff80s} in the 80s that this group is perfect.  However, it has been an open question what occurs in the case $n = 2$.  It turns out that there are two cases: the case where $\mu$ has finite area, and the case where $\mu$ has infinite area.  (Any choice of $\mu$ gives a group that is isomorphic to one of these two cases.)  In the finite area case, it was recently shown in \cite{simp2} that, in contrast to the higher-dimensional case the group is not perfect.  The proof uses PFH spectral invariants: the idea is to view the disc as the complement of the north pole $p_+$ in the two-sphere, and then any diffeomorphism extends to an area-preserving homeomorphism by fixing $p_+$; one can then apply the aforementioned PFH spectral invariants to these homeomorphisms.  In particular, to show that the group is not perfect, one can consider the restriction of an infinite twist 
from above
about $p_+$: if the restriction was a product of commutators, then this infinite twist would be as well, however it is known that every normal subgroup of the group of area-preserving homeomorphisms of the two-sphere contains the commutator subgroup.   

\subsubsection*{Spectral recognition}

It is also interesting to ask how much the periods of the periodic orbits for the Reeb flow know about a manifold with a contact form.  This is analogous to a much studied topic in Riemannian geometry, called ``length spectrum rigidity" and one can call this set of periods the {\em action spectrum}.  We discussed a very special case of this in our discussion of contact pseudorotations above, but one can consider other cases.  For example, inspired by known spectral rigidity results for other quantities (e.g. the Laplace spectrum), one can ask about the case of the ball.  For the contact form on the boundary of a ball, all the simple Reeb orbits have the same period.  The paper \cite{withmarco} uses techniques related to the ideas in these notes to show the converse: if a contact form on $S^3$ has the property that all of its simple Reeb orbits have the same period, then the contact form is actually isomorphic to the contact from on the boundary of the ball; \cite{withmarco} also proves more general results of a similar flavor, classifying contact forms, on any three-manifold, such that all simple Reeb orbits have a common period for the associated Reeb flow.  These results also relate to an old question of Eliashberg and Hofer \cite{eliashbergh}, asking to what degree the interior of a Liouville domain determines its boundary.  For example, the ECH spectral invariants of such a four-dimensional domain depend only on its interior, so from this the results in \cite{withmarco} can be used to give a positive answer to this kind of Eliashberg-Hofer question, in the special case of a ball.  For more about the relationship between the interiors of Liouville domains and the spectrum of the Reeb flow on the boundary, we refer the reader to \cite{stabilityaction}.











\section{Discussion and remarks on higher dimensions}

To start to wrap up these notes, let me share a few ruminations about the content of these notes from a more bird-eyed view. 

I hope I have conveyed with these notes and the accompanying lectures that for these low-dimensional dynamical problems of a symplectic nature, there is quite a large amount of structure.  Here are a few themes that I consider important that I hope I have illustrated about low-dimensional symplectic dynamics:
\begin{enumerate}
\item Symplectic Weyl laws are very powerful
\item There is an abundance of $J$-holomorphic curves and this is very useful.  One can leverage Seiberg-Witten invariants to produce such curves.  
\item Global surface of sections methods have great reach 
\end{enumerate}

To put this in context, perhaps the following question is useful: {\em does any of this work in higher dimensions?}

To start to say something about this, let us note that the questions we have raised in these notes certainly have analogues in higher dimensions.  For example, the Weinstein conjecture still remains open (though certain special cases have been worked out.)  One could ask various kinds of {\em n or infinity} questions.  For example, it is a very old problem whether a star-shaped hypersurface in $\mathbb{R}^{2n}$ has $n$ or infinitely many simple Reeb orbits.  One can also ask, in analogy with the Simplicity Conjecture, whether the group of (compactly supported) {\em symplectic homeomorphisms} of the disc $D^{2n}$ is simple; the $n= 1$ case of this is precisely the simplicity conjecture, but in higher dimensions it is completely open.  (A symplectic homeomorphism is a $C^0$-limit of symplectic diffeomorphisms.)  One could similarly ask about the Le Calvez - Yoccoz property, the closing lemma, etc.   There are some very interesting partial results (see e.g. \cite{cinelietal, cetal}) about these topics in higher dimensions, but for the most part these questions remain wide open.

It seems conceivable that some of the above tools have analogues in higher dimensions.  For example, perhaps there are analogous Weyl laws yet to be discovered.  In our notes, we made heavy use of Seiberg-Witten theory, but as we explained in the previous section, it is now the case that one can prove the simplicity conjecture, for example, solely with pseudoholomorphic curves and without using any Floer theory at all.  On the other hand, it also seems conceivable that there is just more flexibility in higher dimensions.  Or, perhaps there is still much rigidity but one needs new tools beyond pseudoholomorphic curves.  For example, as far as the author is aware, it is hard to imagine recovering the volume constraint for embeddings of $2n$-dimensional symplectic ellipsoids, using (nonnegative dimensional families) of pseudoholomorphic curves in standard setups.  
 Much remains mysterious.

\subsection{Open questions}

Let us close with a few open questions about these kind of low-dimensional questions that I find interesting.  There are many, so I tried to choose a few representative ones.

The first is natural in view of Theorem~\ref{thm:hwz}.

\begin{question}
\label{que:twoinf}
Does every contact form on a closed connected three-manifold have two or infinitely many simple Reeb orbits?
\end{question}

In other words, one would like to remove the torsion assumption from Theorem~\ref{thm:hwz}.  If one wants to use the approach used to prove Theorem~\ref{thm:hwz}, one has to better understand what kind of control we have over the topology of ECH $U$-map curves without the assumption on the Chern class.  The only place where this assumption is used in the proof is to handle the Chern class term in the difference \eqref{eqn:diff} between $I$ and $J_0$.        
Presumably, success here would also allow one to extend the Le Calvez - Yoccoz property to arbitrary Reeb flows on three-manifolds.  As we explained, in the nondegenerate case the two or infinity dichotomy does hold without any torsion assumption, see Theorem~\ref{thm:ley}.

For that matter, one can ask the following:

\begin{question}
Does the Reeb flow associated to a contact form on a closed three-manifold always have a global surface of section? 
\end{question}

One expects that this should imply a positive answer to Question~\ref{que:twoinf}.  Evidence that this could hold is furnished by \cite{CDHR,cm}.  A very recent interesting result proves this for geodesic flows on positive genus surfaces \cite{ma}.

When there are infinitely many orbits, it is also natural to try to understand the growth rate.  For example, here is one question, asked by U. Hryniewicz; Hryniewicz conjectures that it has a positive answer.

\begin{question}
Let $\lambda$ be a contact form on $S^3$ with infinitely many simple Reeb orbits.  Let $f(T)$ denote the number of Reeb orbits with action $\le T$.  Does this grow at least quadratically in $T$?
\end{question}

It may be that the question is not only true for $S^3$ but for any three-manifold.  For geodesic flows on $S^2$, the question was recently answered in the affirmative by the doctoral thesis of B. Albach. 

Another interesting question involves the kind of periodic orbits that can be appear.  For example, as far as I know the following is open:

\begin{question}
Does every contact form on $S^3$ have an elliptic orbit?
\end{question}

Here, we are counting a degenerate orbit as elliptic (or else the question would certainly have a negative answer.)  For more about the dynamical implications of an elliptic orbit, including relationships to structural stability, see \cite{N}.

There are also interesting further questions about generic dynamics.  Here are two samples:

\begin{question}
\begin{enumerate}[(1)]
\item Does a $C^{\infty}$-generic Riemannian metric on a surface have the property that its closed geodesics are dense in the unit tangent bundle?
\item Let $f$ be an area-preserving diffeomorphism of a closed surface $S$ and let $p$ and $q$ be such that the closure of the forward orbit of $p$ meets the closure of the backwards orbit of $q$.  Can $f$ be perturbed in $C^{\infty}$ such that $p$ and $q$ lie on the same orbit?
\end{enumerate}
\end{question}

The first question above can be thought of as the $C^{\infty}$ closing lemma for geodesic flows.  The challenge here is that genericity is demanded in the space of Riemannian metrics; the question does not follow from the $C^{\infty}$ closing lemma for Reeb flows, because if one takes a contact form on the unit cotangent bundle coming from a Riemannian metric and perturbs it locally, it will generally not come from a Riemannian metric anymore.   The second question is the $C^{\infty}$ version of the``Connecting Lemma" (in the conservative surface setting); in the $C^1$-topology this is a fundamental result of Hayashi \cite{h}.
 
 Much also remains unknown about the structure of the various groups of homeomorphisms that we have discussed.  For example, let $N$ denote the $O(1)$ subleading asymptotics subgroup of $\Homeo_c(D^2,\mu_{std})$ described above.
 
 \begin{question}
 Is the ``subleading asymptotics" subgroup $N$ of $\Homeo_c(D^2,\mu_{std})$ simple?
  \end{question}
 
 If it is, it would follow from known results (see e.g. \cite{simp}) that it is equal to the commutator subgroup of $\Homeo_c(D^2,\mu_{std})$.































.

\bigskip

{\footnotesize

{\sc Dan Cristofaro-Gardiner}

University of Maryland, College Park

{\em dcristof@umd.edu}


\end{document}

\bibliographystyle{alpha}
\bibliography{main}

\end{document}